\documentclass[11pt,oneside,english]{amsart}
\usepackage[T1]{fontenc}
\usepackage[latin9]{inputenc}
\usepackage{geometry}
\usepackage{amsbsy}
\usepackage{amstext}
\usepackage{amsthm}
\usepackage{amssymb}

\makeatletter
\numberwithin{equation}{section}
\numberwithin{figure}{section}
\theoremstyle{plain}
\newtheorem{thm}{\protect\theoremname}[section]
\theoremstyle{remark}
\newtheorem{rem}[thm]{\protect\remarkname}
\theoremstyle{plain}
\newtheorem{cor}[thm]{\protect\corollaryname}
\theoremstyle{plain}
\newtheorem{conjecture}[thm]{\protect\conjecturename}
\theoremstyle{plain}
\newtheorem{lem}[thm]{\protect\lemmaname}
\theoremstyle{definition}
\newtheorem{example}[thm]{\protect\examplename}
\theoremstyle{plain}
\newtheorem{prop}[thm]{\protect\propositionname}
\theoremstyle{definition}
\newtheorem{defn}[thm]{\protect\definitionname}

\def\makebbb#1{
    \expandafter\gdef\csname#1\endcsname{
        \ensuremath{\Bbb{#1}}}
}\makebbb{R}\makebbb{N}\makebbb{Z}\makebbb{C}\makebbb{H}\makebbb{E}\makebbb{H}\makebbb{P}\makebbb{B}\makebbb{Q}\makebbb{E}\makebbb{E}
\makebbb{H}
\usepackage{babel}

\usepackage{babel}

\makeatother

\usepackage{babel}

\makeatother

\usepackage{babel}
\providecommand{\conjecturename}{Conjecture}
\providecommand{\corollaryname}{Corollary}
\providecommand{\definitionname}{Definition}
\providecommand{\examplename}{Example}
\providecommand{\lemmaname}{Lemma}
\providecommand{\propositionname}{Proposition}
\providecommand{\remarkname}{Remark}
\providecommand{\theoremname}{Theorem}

\begin{document}
\title{An invitation to Kähler-Einstein metrics and random point processes }
\author{Robert J. Berman}
\begin{abstract}
This is an invitation to the probabilistic approach for constructing
Kähler-Einstein metrics on complex projective algebraic manifolds
$X.$ The metrics in question emerge in the large $N-$limit from
a canonical way of sampling $N$ points on $X,$ i.e. from random
point processes on $X,$ defined in terms of algebro-geometric data.
The proof of the convergence towards Kähler-Einstein metrics with
negative Ricci curvature is explained. In the case of positive Ricci
curvature a variational approach is introduced to prove the conjectural
convergence, which can be viewed as a probabilistic constructive analog
of the Yau-Tian-Donaldson conjecture. The variational approach reveals,
in particular, that the convergence holds under the hypothesis that
there is no phase transition, which - from the algebro-geometric point
of view - amounts to an analytic property of a certain Archimedean
zeta function.
\end{abstract}

\address{Robert J. Berman, Mathematical Sciences, Chalmers University of Technology
and the University of Gothenburg, SE-412 96 Göteborg, Sweden}
\email{robertb@chalmers.se}
\dedicatory{Dedicated to Shing-Tung Yau on the occasion of his 70th birthday. }
\maketitle

\section{Introduction}

A central theme in current complex geometry is the interplay between
the differential geometry of Kähler-Einstein metrics on a compact
complex manifold $X$ and complex algebraic geometry. These connections
were emphasized by Yau in the 80s \cite{y2}, leading up to the formulation
of the Yau-Tian-Donaldson conjecture involving the notion of K-stability
\cite{ti,do1}. This paper is as an invitation to the \emph{probabilistic
}approach to Kähler-Einstein metrics, initiated in \cite{berm8,berm8 comma 5,berm6}.
From this probabilistic point of view the Kähler-Einstein metrics
emerge in the large $N-$limit from a canonical way of sampling $N$
points on $X,$ i.e. from canonical random point process on $X.$
The point processes are defined in terms of algebro-geometric data
and furnish canonical and explicit approximations to Kähler-Einstein
metrics, expressed as period integrals. The thrust of this approach
is thus that it provides a new explicit bridge between Kähler-Einstein
metrics and algebraic geometry. Moreover, it leads to a new notion
of stability, dubbed Gibbs stability, which naturally fits into the
logarithmic setup of the Minimal Model Progrom of current birational
algebraic geometry. The investigation of the large $N-$limit of the
point processes also naturally ties in with the variational approach
to Kähler-Einstein metrics in \cite{bbgz,bbegz,bbj} (via the notion
a large devition principle). 

The aim of the present work is to provide both a survey of the probabilistic
construction of Kähler-Einstein metrics with \emph{negative} Ricci
curvature in \cite{berm8}, as well as a variational approach towards
a proof of the remaining main open problem concerning\emph{ positive
}Ricci curvature. As demonstrated in \cite{berm11} this approach
settles the problem in the case of log Fano curves, but in the general
case it hinges an a conjectural energy bound. Interestingly, as explained
in \cite{berm11}, already the simplest case of a log Fano curve with
trivial automorphism group - the complex projective line decorated
with $3$ weighted points - exhibits some rather intruiging connections
between this probabilistic approach and the classical theory of hypergeometric
functions and integrals, as well as Conformal Field Theory and Integrable
Systems. 

There are also many other connections to other fields that will not
be covered here. Motivations from Quantum Gravity are discussed in
\cite{berm5}, where a heuristic argument for the large $N-$convergence
was first given. The connections to pluripotential theory and interpolation
theory in complex affine space $\C^{n}$ are covered in \cite{berm10b,du}
and relations to stochastic gradient flows and optimal transport (via
tropicalization) appear in \cite{berm10a}. Moreover, connections
to Arithmetic Geometry and Non-Archimidean Geometry will be developed
elsewhere, as well as relations to the AdS/CFT correspondence in theoretical
physics, connecting geometry to Conformal Field Theory \cite{b-c-p}. 

Finally, a caveat: ``The probabilistic approach to Kähler-Einstein
metrics'' will, in the following, refer to the approach, based on
random point processes with $N$ points, introduced in \cite{berm5,berm8,berm8 comma 5}.
It should, however, be pointed out that there is also a different
probabilistic approach to Kähler geometry, introduced in \cite{f-k-z},
using random elements in the symmetric spaces $SL(N,\C)/SU(N).$ In
the case of a a complex curve, there are also, as explained in \cite{berm11},
some relations between the canonical random point processes and the
theory of Gaussian Multiplicative Chaos, as used in the probabilistic
approach to Liouville Quantum Gravity in \cite{K-R-V}. It should
also be pointed out that a statistical mechanics approach has previously
been applied to conformal geometry in \cite{k2}, which is related
to the present complex-geometric setup in the case of complex curves.

\tableofcontents{}

\subsection{Organization}

Since this work is mainly aimed to be expository, most of the proofs
are merely sketched. However, there are also some new results (Theorems
\ref{thm:lsc intro}, \ref{thm:log curve intro} , \ref{thm:Fano setting text},
\ref{thm:phase} ) where more details are provided. 

After a brief recap of Kähler-Einstein metrics in Section \ref{sec:Recap-of-K=0000E4hler-Einstein}
and some motivation we give, in Section \ref{sec:A-bird's-eye-view},
a bird's-eye view of the probabilistic approach to Kähler-Einstein
metrics. First, the main results concerning the case when the sign
$\beta$ of the canonical line bundle $K_{X}$ is positive are explained.
Then, the conjectural picture concerning the Fano setting, where $\beta<0,$
is described. In Section \ref{sec:Background} background material
is provided on complex geometry, pluripotential theory, probability
and variational analysis. This material is needed for the more detailed
view of the probabilistic approach, which is the subject of the subsequent
sections, starting in Section \ref{sec:The-thermodynamical-formalism}
with a thermodynamical formalism for Kähler-Einstein metrics. The
focus in this section is on the analytical properties of the free
energy functional (which appears as the rate functional of a large
deviation principle for the canonical point processes). The connection
to the standard functionals in Kähler geometry, in particular the
Mabuchi functional is also explained. In the following Section \ref{sec:The-case-of pos beta}
we explain the proof of the convergence of the point processes and
the general large deviation principle in the case $\beta>0.$ The
Fano setting, where $\beta<0,$ is considered in Section \ref{sec:Towards-the-case of neg}
where a a variational approach is proposed for proving the conjectural
convergence of the point processes and relations to phase transitions
are explored. 

\subsection{Acknowledgements}

I am greatful to Shing-Tung Yau for the invitation to contribute to
the upcoming volume of Surveys in Differential Geometry, on the occasion
of his 70th birthday - his work has been a constant source of inspiration.
It is also a pleasure to thank Sébastien Boucksom, David Witt-Nyström,
Vincent Guedj and Ahmed Zeriahi for the stimulating collaborations
\cite{b-b,b-b-w,bbgz}, which paved the way for the current probabilistic
approach. Also thanks to Daniel Persson and the referee for comments
on the paper. This work was supported by grants from the KAW foundation,
the Göran Gustafsson foundation and the Swedish Research Council. 

\section{\label{sec:Recap-of-K=0000E4hler-Einstein}Recap of Kähler-Einstein
metrics and motivation}

Recall that a Kähler metric $\omega$ on a compact complex manifold
$X$ of dimension $n$ is said to be \emph{Kähler-Einstein} if it
has constant Ricci curvature:

\[
\text{Ric \ensuremath{\omega}=\ensuremath{-\beta}\ensuremath{\omega}}
\]
for some constant $\beta.$ The existence of such a metric implies
that the\emph{ canonical line bundle $K_{X}$} of $X$ (i.e. the top
exterior power of the cotangent bundle of $X)$ has a definite sign:
\begin{equation}
\text{sign}(K_{X})=\text{sign}(\beta)\label{eq:sing of K X}
\end{equation}
 We will be using the standard terminology of positivity in complex
geometry: a line bundle is said to be \emph{positive,} $L>0,$ if
it is ample and \emph{negative,} $L<0,$ if its dual. Moreover, we
shall adopt the standard additive notation for tensor products of
line bundles, so that the dual of $L$ is expressed as $-L$ (see
Section \ref{subsec:Metrics-on-line}). Here we will focus on the
case when $\beta\neq0$ (see \cite[Section 6.1]{berm8} for probabilistic
aspects of the case $\beta=0).$ Then $X$ is automatically a complex
projective algebraic manifold. After a rescaling of the metric we
may as well assume that $\beta=\pm1.$ In the case when $K_{X}>0$
and $K_{X}=0$ the existence of a Kähler-Einstein metric was established
in the seminal works \cite{y,au} and \cite{y}, respectively. However,
in the case when $K_{X}<0,$ i.e. when $X$ is a \emph{Fano manifold,}
there are obstructions to the existence of a Kähler-Einstein metric.
According to the \emph{Yau-Tian-Donaldson conjecture} a Fano manifold
$X$ admits a Kähler-Einstein metric iff $X$ is K-polystable (the
``only if'' statement holds for singular Fano varieties \cite{berman6ii}).
This an algebro-geometric notion of stability, modeled on the notion
of stability in Geometric Invariant Theory. Briefly, a Fano manifold
$X$ is K-polystable iff the Donaldson-Futaki invariant $DF(\mathcal{X})$
of any normal $\C^{*}-$equivariant deformation $\mathcal{X}\rightarrow\C$
of $(X,-K_{X})$ (called a test configuration) is non-negative and
vanishes only for the test configurations $\mathcal{X}$ whose central
fiber is biholomorphic to $X.$ The invariant $DF(\mathcal{X})$ may
be defined as a the normalized large $N-$limit of the Chow weight
of the orbit of $X$ in $\P^{N-1}$ under a one-parameter subgroup
of $SL(N,\C)$ associated to $\mathcal{X}.$ See the survey \cite{do2}
for the detailed definition and further background. 

When a Fano manifold $X$ admits a Kähler-Einstein metric it is uniquely
determined modulo the action of the group $\text{Aut \ensuremath{(X)_{0}}}$
of all biholomorphic automorphisms of $X$ in the connected component
of the identity \cite{b-m}. The dichotomy non-trivial vs. trivial
group $\text{Aut \ensuremath{(X)_{0}}}$ is reflected in the difference
between \emph{K-polystability }and the stronger notion of \emph{K-stability,
}which implies that $\text{Aut \ensuremath{(X)_{0}}}=\{I\}.$

The relation between Kähler-Einstein metrics and stability, in the
sense of GIT, was propounded by Yau in \cite{y2,y3} and further developed
by Tian \cite{ti} and then Donaldson \cite{do1}, who considered
the more general setting of constant scalar curvature metrics on polarized
complex compact manifolds. In the Fano case the Yau-Tian-Donaldson
conjecture was settled in \cite{c-d-s}, using a logarithmic version
of Aubin's continuity (and the Cheeger-Colding-Tian theory of Gromov-Hausdorff
limits of Kähler manifolds). There is also a stronger form of K-stability,
called\emph{ uniform K-stability}, introduced in \cite{sz1,sz2,de,bhj}.
In \cite{bbj} it was shown by a direct variational approach that
a Fano manifold admits a unique Kähler-Einstein -metric iff $X$ is
uniformly K-stable, using pluripotential theory (including some Non-Archimidean
aspects). Very recently the variational approach in \cite{bbj} has
also been extended to general singular Fano varieties in \cite{l-t-w}
and \cite{li2}, using a notion of equivariant uniform K-stability.

\subsection{The lack of explicit formulas}

Already in the case when the canonical line bundle $K_{X}$ is positive
and thus $X$ admits a Kähler-Einstein metric $\omega_{KE}$ with
negative Ricci curvature, there are very few cases where $\omega_{KE}$
can be written down explicitly. For example, let $X$ be the projective
algebraic manifold defined by the zero-locus of a homogeneous polynomial
of degree $d$ in $(n+1)-$dimensional complex projective space $\P_{\C}^{n+1}.$
Then $K_{X}>0$ iff $d>(n+2)$ (by the adjunction formula). However,
even in the case of a complex algebraic curve $X$ in $\P_{\C}^{2}$
the problem of explicitly describing the Kähler-Einstein metric on
$X$ is, in general, intractable. By the classical uniformization
theory this problem is equivalent to finding an explicit biholomorphic
map from $X$ to the quotient $\H/G$ of the upper half-plane by a
discrete subgroup $G\subset SL(2,\R).$ This has only been achieved
for very special curves, using techniques originating in the classical
works by Weierstrass, Riemann, Fuchs, Schwartz, Klein, Poincaré,...\cite{s-g}.
See, for example, \cite{b-g} for the case when $X$ is a Fermat curve
of degree $d\geq4,$ \cite{el} for the case when $X$ is the Klein
quartic, including arithmetic aspects and \cite{do} for connections
to the mirror-moonshine conjecture of Lian-Yau. A recurrent theme
in all these cases is that the uniformizing map from $X$ to $\H/G$
may be expressed as a quotient of hypergeometric functions. 

Thus one motivation for the probabilistic approach to Kähler-Einstein
metrics is that for any given projective manifold $X$ with $K_{X}>0$
it leads to canonical approximations $\omega_{k}$ of the Kähler-Einstein
metric $\omega_{KE}$ on $X,$ which are explicitly expressed in terms
of algebro-geometric data on $X.$ This will be explained in detail
in the following sections (see Corollary \ref{cor:gen type conv of correl measu intro}).
For the moment we just point out that, realizing $X$ as an algebraic
subvariety of $\P_{\C}^{m},$ the Kähler potential $\varphi_{k}$
on $X$ may be explicitly expressed as follows, in terms of homogeneous
coordinates $z\in\C^{m+1}$ on $\P_{\C}^{m}:$ 
\begin{equation}
e^{\varphi_{k}(z)}=\frac{1}{Z_{N_{k}}}\int_{X^{N_{k}-1}}\frac{\left|P_{N_{k}}(z,z_{2}...,z_{N_{k}})\right|^{2/k}}{\left|z\right|^{2}\left|z_{2}\right|^{2}\cdots\left|z_{N_{k}}\right|^{2}}dV^{\otimes(N_{k}-1)}\label{eq:formula for phi k}
\end{equation}
 where $N_{k}$ is a sequence tending to infinity (the plurigenera
of $X$), $P_{N_{k}}$ is a homogeneous polynomial on the $N_{k}-$fold
product $X^{N_{k}}$ of $X,$ naturally attached to the degree $k$
component of graded homogeneous coordinate ring of $X$ and $dV$
is an algebraic volume form on $X,$ induced by the embedding in $\P^{m}.$
The normalizing constant $Z_{N_{k}}$ is given by $Z_{N,\beta}$ for
$\beta=1$ where $Z_{N,\beta}$ is the following Archimedean zeta
function\footnote{Integrals of complex powers of algebraic functions are Archimedean
analogs of local Igusa zeta functions, defined in a p-adic setting
\cite{ig}. } 
\begin{equation}
Z_{N,\beta}:=\int_{X^{N_{k}}}\left(\frac{\left|P_{N_{k}}(z_{1},z_{2}...,z_{N_{k}})\right|^{2/k}}{\left|z_{1}\right|^{2}\left|z_{2}\right|^{2}\cdots\left|z_{N_{k}}\right|^{2}}\right)^{\beta}dV^{\otimes N_{k}}.\label{eq:Z N beta in motivation}
\end{equation}
In the case when $X$ is Fano the exponent $2/k$ in formula \ref{eq:formula for phi k}
is replaced by $2\beta/k$ for $\beta=-1.$ The conjectural convergence
towards the potential of a Kähler-Einstein metric, when $N_{k}\rightarrow\infty,$
then turns out to be related to analytic properties of the corresponding
Archimedean zeta function $\beta\mapsto Z_{N,\beta}.$ 

Incidentally, the integral in formula \ref{eq:Z N beta in motivation}
is reminiscent of the Euler integral (period) formulas for the hypergeometric
functions that appear in the classical uniformization theory for complex
curves and integrable systems, alluded to a above. Indeed, the ordinary
(Gauss) hypergeometric function with real parameters $(a,b,c)$ may
be expressed as follows when $z$ is in the upper half-plane of $\C:$
\begin{equation}
F(z)=\frac{1}{B(b,c-b)}\int_{[0,1]}(1-zx)^{-a}x^{b-1}(1-x)^{c-b-1}dx,\label{eq:hypergeom}
\end{equation}
 where $B(p,q)$ is the Beta function: 
\begin{equation}
B(p,q)=\int_{[0,1]}x^{p-1}(1-x)^{q-1}dx,\,\,\,\,\text{Re \ensuremath{p>0,\,\,\text{Re }q>0}}\label{eq:def of Beta f}
\end{equation}
These connections can be made more precise in the case of complex
curves, as explained in \cite{berm11}.
\begin{rem}
When $X$ is defined over $\Q$ the integrals in \ref{eq:Z N beta in motivation}
and \ref{eq:hypergeom} are examples of\emph{ periods}, as defined
in \cite[Chapter 4]{k-z}, i.e. integrals of the form $\int_{\gamma}\eta$
for an algebraic form $\eta$ of maximal degree on a projective variety
$Y$ defined over $\Q$ and homology class $\gamma\in H(Y(\C),\Q)$
with boundary on a normal crossing divisor in $Y(\C).$ Indeed, $Y(\C)$
can be taken as the complexification of $X,$ when $X$ is viewed
as a real manifold of real dimension $2n$ and $\gamma$ as $X.$
As stressed in \cite{k-z}, in many classical cases period integrals
satisfy differential equations with respect to variations of the parameters
of the integrand (notably Picard-Fuch equations, such as the hypergeometric
equation satisfied by $F(z)$ above). In the present setup the role
of the differential equation is thus - loosely speaking - played by
the Kähler-Einstein equation, but it only arises in the limit when
$N\rightarrow\infty.$ 
\end{rem}

The canonical Kähler potentials $\varphi_{k}$ in formula \ref{eq:formula for phi k}
are also somewhat reminiscent of the sequence of balanced metrics
introduced in \cite{do1c} (when either $K_{X}>0$ or $K_{X}<0),$
defined as fixed point of an algebraic iteration. But one virtue of
the present setup is that $\varphi_{k}$ is given by an explicitly
formula.

\section{\label{sec:A-bird's-eye-view}A bird's-eye view of the probabilistic
approach to Kähler-Einstein metrics}

We continue with the notation introduced in Section \ref{sec:Recap-of-K=0000E4hler-Einstein},
with $X$ denoting a compact complex manifold with the property that
its canonical line bundle $K_{X}$ has a definite sign $\beta.$ First
recall that it follows directly from the basic formula for the Ricci
curvature of a Kähler metric, that, in the case when $\beta\neq0,$
a Kähler-Einstein metric $\omega_{KE}$ on $X$ can be recovered from
its (normalized) volume form $dV_{KE}$
\[
\omega_{KE}=\frac{1}{\beta}dd^{c}\log dV_{KE},
\]
 using the notation $dd^{c}:=\frac{i}{2\pi}\partial\bar{\partial}$
(see Section \ref{subsec:Complex-geometry}). Accordingly, the strategy
of the probabilistic approach is to first construct the canonical
normalized volume form $dV_{KE}$ by a canonical sampling procedure
on $X.$ 

We will denote by $\mathcal{P}(X)$ the space of all normalized measures
$\mu$ on $X,$ i.e. the space of all probability measures on $X$
(see Section \ref{subsec:Probabilistic-setup} for a recap of the
basic probabilistic background). 

\subsection{The case $\beta>0$}

Let $X$ be a compact complex manifold with positive canonical line
bundle $K_{X}.$ The starting point of the probabilistic approach
to Kähler-Einstein metrics is the observation that, when $K_{X}>0,$
there is\emph{ canonical} way of sampling configurations of $N$ random
points on $X,$ i.e. there is a canonical \emph{random point process}
on $X$ with $N$ points. This means that there is a canonical sequence
of symmetric probability measures $\mu^{(N)}$ on the $N-$fold products
$X^{N}$ and, as shown in \cite{berm8}, the corresponding \emph{empirical
measure }
\[
\delta_{N}:=\frac{1}{N}\sum_{i=1}^{N}\delta_{x_{i}}:\,\,X^{N}\rightarrow\mathcal{P}(X)
\]
 viewed as a random measure on the probability space $(X^{N},\mu^{(N)}),$
converges in probability as $N\rightarrow\infty,$ to the volume form
$dV_{KE}$ of the unique Kähler-Einstein metric $\omega_{KE}.$ In
fact, the canonical sequence of probability measures on $\mu^{(N)}$
on $X^{N_{k}}$ is defined for a specific subsequence of integers
$N$ tending to infinity, the \emph{plurigenera }of $X:$ 
\[
N_{k}:=\dim H^{0}(X,kK_{X}),
\]
 where $H^{0}(X,kK_{X})$ denotes the complex vector space of all
global holomorphic sections $s^{(k)}$ of the $k$ th tensor power
of the canonical line bundle $K_{X}\rightarrow X$ (aka pluricanonical
forms). We recall that, in terms of local holomorphic coordinates
$z\in\C^{n}$ on $X,$ this simply means that a section $s^{(k)}$
may be represented by local holomorphic functions $s^{(k)}$ on $X,$
such that $|s^{(k)}|^{2/k}$ transforms as a density on $X,$ i.e.
defines a measure on $X.$ The canonical probability measure $\mu^{(N_{k})}$
on $X^{N_{k}}$ is now defined by 
\begin{equation}
\mu^{(N_{k})}:=\frac{1}{Z_{N_{k}}}\left|\det S^{(k)}\right|^{2/k},\,\,\,\,Z_{N}:=\int_{X^{N_{k}}}\left|\det S^{(k)}\right|^{2/k}\label{eq:canon prob measure intro}
\end{equation}
 where $\det S^{(k)}$ is the holomorphic section of the canonical
line bundle $(kK_{X^{N_{k}}})$ over $X^{N_{k}}$, defined by the
Slater determinant
\[
(\det S^{(k)})(x_{1},x_{2},...,x_{N}):=\det(s_{i}^{(k)}(x_{j})),
\]
 in terms of a given basis $s_{i}^{(k)}$ in $H^{0}(X,kK_{X}).$ Note
that under a change of bases the section $\det S^{(k)}$ only just
changes by a multiplicative complex constant $c$ (the determinant
of the change of bases matrix). Hence, by homogeneity, the probability
measure $\mu^{(N_{k})}$ is independent of the choice of bases and
thus defines a canonical symmetric probability measure on $X^{N_{k}},$
as desired. Moreover, the probability measure is $\mu^{(N_{k})}$
is encoded by algebro-geometric data in the following sense: using
the Kodaira embedding theorem to realize $X$ as projective algebraic
subvariety the density $\det S^{(k)}$ can be identified with a homogeneous
polynomial (see Section \ref{subsec:Projectiv-algebraic-embeddings}). 

The following convergence result was shown in \cite{berm8}:
\begin{thm}
\label{thm:ke intro}Let $X$ be a compact complex manifold with positive
canonical line bundle $K_{X}.$ Then the empirical measures $\delta_{N_{k}}$
of the corresponding canonical random point processes on $X$ converge
in probability, as $N_{k}\rightarrow\infty,$ towards the normalized
volume form $dV_{KE}$ of the unique Kähler-Einstein metric $\omega_{KE}$
on $X.$ 
\end{thm}

In particular, the convergence in probability in the previous theorem
implies that the measure on $X$ defined by the expectations $\E(\delta_{N_{k}})$
of the empirical measure $\delta_{N_{k}}$ converge towards $dV_{KE}$
in the weak topology of measures on $X:$ 
\[
\E(\delta_{N_{k}})=\int_{X^{N_{k}-1}}\mu^{(N_{k})}\rightarrow dV_{KE},\,\,\,k\rightarrow\infty
\]
Noting that 
\begin{equation}
\omega_{k}:=dd^{c}\log\E(\delta_{N_{k}})=dd^{c}\log\int_{X^{N_{k}-1}}\left|\det S^{(k)}\right|^{2/k}\label{eq:def of omega k KX pos intro}
\end{equation}
defines a canonical sequence of Kähler metrics on $X$ (for $k$ sufficiently
large) we thus arrive at the following 
\begin{cor}
\label{cor:gen type conv of correl measu intro}Let $X$ be a compact
complex manifold with positive canonical line bundle $K_{X}.$ Then
the canonical sequence of Kähler metrics $\omega_{k}$ converges towards
the unique Kähler-Einstein metric $\omega_{KE}$ on $X,$ in the weak
topology.
\end{cor}

It does not seem clear how to directly prove the convergence of the
measures $\E(\delta_{N_{k}})$ on $X$ (and the Kähler metrics $\omega_{k})$
without first proving the stronger convergence in probability in Theorem
\ref{thm:ke intro}. Moreover, as explained in Section \ref{sec:The-case-of pos beta},
the convergence in probability is shown to hold in the stronger exponential
sense of a Large Deviation Principle (LDP). Endowing the space $\mathcal{P}(X)$
of all probability measures on $X$ by a metric compatible with the
weak topology and denoting by $B_{\epsilon}(\mu)$ the ball of radius
$\epsilon$ centered at a given element $\mu\in\mathcal{P}(X)$ this
essentially means that there exists a functional $I(\mu)$ on $\mathcal{P}(X),$
called the\emph{ rate functional,} such that
\begin{equation}
\text{Prob }\left(\frac{1}{N}\sum_{i=1}^{N}\delta_{x_{i}}\in B_{\epsilon}(\mu)\right)\sim e^{-NI(\mu)}\label{eq:LDP prob balls intro}
\end{equation}
when first $N\rightarrow\infty$ and then $\epsilon\rightarrow0$
(see Section \ref{subsec:LDP} for the precise meaning of a LDP).
Moreover, the\emph{ }rate functional $I(\mu)$ is non-negative and
vanishes iff $\mu=dV_{KE}$ In fact, if $\omega$ is a Kähler form
in the first Chern class $c_{1}(K_{X}),$ then the rate functional
$I$ may be identified with the the Mabuchi  functional on the space
of Kähler metrics: 
\[
I(\frac{\omega^{n}}{\int_{X}\omega^{n}})=\mathcal{M}(\omega),
\]
 where $\mathcal{M}(\omega)$ denotes Mabuchi  functional on the space
of Kähler metrics in $c_{1}(K_{X}).$ Once the LDP \ref{eq:LDP prob balls intro}
has been established the convergence in probability in Theorem \ref{thm:ke intro}
then follows from the well-known fact in Kähler geometry that $\omega_{KE}$
is unique minimizer of $\mathcal{M}.$ More precisely, since the rate
functional $I(\mu)$ is defined on the whole space of probability
measure $\mathcal{P}(X)$ and not only on the dense subspace of volume
forms, some additional arguments are required, using variational calculus
on $\mathcal{P}(X)$ (see Theorem \ref{thm:free energy beta pos}). 

\subsubsection{\label{subsec:The-proof-of LDP intro}The proof of the LDP for $\beta>0$
and statistical mechanics }

Fixing a volume form $dV$ on $X$ the canonical probability measure
\ref{eq:canon prob measure intro} may be expressed as 

\[
\mu^{(N)}=\frac{1}{Z_{N_{k}}}\left\Vert \det S^{(k)}\right\Vert ^{2/k}dV^{\otimes N},
\]
 where $\left\Vert \cdot\right\Vert $ denotes the metric on $K_{X}$
induced by the volume form $dV.$ The starting the point of the proof
of the LDP \ref{eq:LDP prob balls intro} is to rewrite this expression
 as a \emph{Gibbs measure:}
\begin{equation}
\mu_{\beta}^{(N_{k})}=\frac{e^{-\beta NE^{(N)}}}{Z_{N,\beta}}dV^{\otimes N},\,\,\,Z_{N,\beta}:=\int_{:X^{N}}e^{-\beta NE^{(N)}}dV^{\otimes N}\label{eq:def of Gibbs measure intro}
\end{equation}
 with 
\begin{equation}
E^{(N)}:=-\frac{1}{kN}\log\left\Vert \det S^{(k)}(x_{1},...,x_{N_{k}})\right\Vert ^{2},\,\,\,\beta=1\label{eq:def of E N intro}
\end{equation}
In the general terminology of statistical mechanics, if $X$ is a
Riemannian manifold (where we assume that the Riemannian volume form
$dV$ is normalized) and $E^{(N)}$ is a given symmetric function
on $X^{N},$ called the \emph{energy per particle}, the Gibbs measure
\ref{eq:def of Gibbs measure intro} represents the microscopic equilibrium
state of $N$ interacting identical particles on $X$  at inverse
temperature $\beta.$ The normalizing constant $Z_{N,\beta}$ is called
the \emph{partition function.}

The proof of the LDP is inspired by the notion of a\emph{ mean field
approximations} in physics. Briefly, the idea is to first show that
there exists a functional $E$ on $\mathcal{P}(X)$ such that the
energy per particle, may be approximated as
\[
E^{(N)}(x_{1},...x_{N})\approx E(\mu),\,\,\,\,\frac{1}{N}\sum_{i=1}^{N}\delta_{x_{i}}\approx\mu
\]
in an appropriate sense, as $N\rightarrow\infty.$ Formally, this
suggests that the rate functional is given by 
\[
I(\mu)=F_{\beta}(\mu)-\inf_{\mathcal{P}(X)}F_{\beta},\,\,\,F_{\beta}(\mu)=\beta E(\mu)+D_{dV}(\mu)\in]0,\infty],
\]
 where $D_{dV}(\mu)$ is the entropy of $\mu$ relative to $dV,$
which arises when integrating the volume $dV^{\otimes N}$ form over
small balls in the $N-$particle configuration space $X^{N}/S_{N}$
(see Section \ref{subsec:The-proof-of Thm LDP text}).

In the general statistical mechanical setup $E(\mu)$ represents the\emph{
energy }of the macroscopic state $\mu$ and $F_{\beta}(\mu)$ its
\emph{free energy,} at inverse temperature $\beta.$ In the present
setting the role of the macroscopic energy $E(\mu)$ is played by
the \emph{pluricomplex energy }of the measure $\mu$ (introduced in
\cite{bbgz}), defined with respect to the Kähler form $-\text{Ric \ensuremath{dV}}.$
More generally, the same proof yields the following general result,
where the role of $K_{X}$ is played by a given positive line bundle
$L$ over a compact complex manifold $X.$
\begin{thm}
\label{thm:LDP intro}Let $L$ be a positive line bundle over a compact
complex manifold $(X,L).$ Given a volume form $dV$ on $X$ and a
metric $\left\Vert \cdot\right\Vert $ on $L$ denote by $\mu_{\beta}^{(N)}$
the corresponding Gibbs measure \ref{eq:def of Gibbs measure intro},
at inverse temperature $\beta\in]0,\infty[.$ Then the laws of the
corresponding random measures $\delta_{N}$ on $(X^{N},\mu_{\beta}^{(N)})$
satisfy a Large Deviation Principle (LDP) with speed $N$ and rate
functional $F_{\beta}(\mu)-C_{\beta},$ where 
\[
F_{\beta}(\mu)=\beta E_{\omega_{0}}(\mu)+D_{dV}(\mu),\,\,\,C_{\beta}=\inf_{\mathcal{P}(X)}F_{\beta}.
\]
\end{thm}

As a consequence, the empirical measure $\delta_{N}$ converges in
probability towards the unique minimizer $\mu_{\beta}$ of the free
energy functional $F_{\beta}$ on $\mathcal{P}(X),$ which is the
volume form characterized by the property that the Kähler metric 
\[
\omega_{\beta}:=\omega_{0}+\frac{1}{\beta}dd^{c}\log\frac{\mu_{\beta}}{dV}
\]
is the unique solution to the\emph{ twisted Kähler-Einstein equation
}
\[
\mbox{\ensuremath{\mbox{Ric}}\ensuremath{\omega}}=-\beta\omega+\theta,\,\,\,\theta:=\beta\omega_{0}+\ensuremath{\mbox{Ric}}\ensuremath{dV}.
\]
 In particular, specializing the previous theorem to the ``canonical
case'' $L=K_{X},$ $\omega_{0}:=-\ensuremath{\mbox{Ric}}\ensuremath{dV}$
and $\beta=1$ yields Theorem \ref{thm:ke intro}. The technical ingredients
in the proof of the LDP in Theorem \ref{thm:LDP intro} are discussed
in Section \ref{sec:The-case-of pos beta}. An important feature of
the case $\beta>0$ is that $\beta E^{(N)}$ is quasi-superharmonic.
However, allowing negative values of $\beta$ is crucial in the case
when $X$ is a Fano manifold, as discussed below.

\subsubsection{Varieties of positive Kodaira dimension and log pairs}

Before turning to Fano manifolds we recall that, as shown in \cite{berm8 comma 5},
Theorem \ref{thm:ke intro} holds in much greater generality. Indeed,
the minimal requirement for the previous setup to apply is that the
plurigenera $N_{k}$ of $X$ tend to infinity as, $k\rightarrow\infty.$
In the classical terminology of algebraic geometry this means that
$K_{X}$ has strictly positive Kodaira dimension $\kappa$ and is
shown in \cite{berm8 comma 5} the analog of Theorem \ref{thm:ke intro}
then holds if $dV_{KE}$ is replaced by the canonical measure on $X$
first introduced in \cite{ts,s-t} in different contexts. Then $\omega_{k}$
(defined by formula \ref{eq:def of omega k KX pos intro}) defines
a canonical sequence of positive currents in $c_{1}(K_{X})$ which
are Kähler on the complement in $X$ of the base locus defined by
$kK_{X}$ and $\omega_{k}$ converges in the weak sense of currents
to the pull-back to $X$ of the canonical twisted Kähler-Einstein
current on the $\kappa-$dimensional base $Y$ of the litaka fibration
$X\rightarrow Y,$ whose generic fibers are Calabi-Yau manifolds of
dimension $n-\kappa$. 

In fact, these results hold more generally when $X$ is a (normal)
projective algebraic variety (with klt singularities) and may be formulated
in a birationally invariant manner using the general the setting of
log pairs $(X,\Delta)$, in the usual sense of MMP \cite{ko}. Recall
that a \emph{log pair} $(X,\Delta)$ is a complex (normal) variety
$X$ endowed with a $\Q-$divisor $\Delta$ on $X,$ i.e. a sum of
irreducible subvarieties of $X$ of codimension one, with coeffients
$w_{i}$ in $\Q.$ Then the role of the canonical line bundle $K_{X}$
is placed by the \emph{log canonical line bundle} 
\[
K_{(X,\Delta)}:=K_{X}+\Delta
\]
(when defined as a $\Q-$line bundle) and the role of the Ricci curvature
$\mbox{Ric \ensuremath{\omega}}$ of a metric $\omega$ is played
by twisted Ricci curvature $\mbox{Ric \ensuremath{\omega}}-[\Delta],$
where $[\Delta]$ denotes the current of integration defined by $\Delta.$
The corresponding \emph{log Kähler-Einstein equation} reads
\begin{equation}
\mbox{Ric \ensuremath{\omega}}-[\Delta]=\beta\omega,\label{eq:log KE eq with beta intro}
\end{equation}
 where $[\Delta]$ denotes the current of integration along $\Delta.$
When $\beta$ is non-zero existence of a solution $\omega_{KE}$ forces
\[
\beta(K_{X}+\Delta)>0
\]
In general, the equation \ref{eq:log KE eq with beta intro} should
be interpreted in the weak sense of pluripotential theory \cite{egz,bbegz}.
However, in the log smooth case it follows from \cite{jmr,g-p} that
a positive current $\omega$ solves the equation \ref{eq:log KE eq with beta intro}
iff $\omega$ is a bona fide Kähler-Einstein metric on $X-\Delta$
and $\omega$ has edge-cone singularities along $\Delta,$ with cone-angle
$2\pi(1-w_{i}),$ prescribed by the coefficents $w_{i}$ of $\Delta.$ 
\begin{rem}
\label{rem:klt}Starting with a normal variety $Y$ such that $K_{Y}$
is defined as a $\Q-$line bundle and taking $X$ to be a non-singular
resolution of $Y,$ $X\rightarrow Y,$ the pull-back of $K_{Y}$ to
$X$ is of the form $K_{(X,\Delta)}$ for an exceptional divisor $\Delta$
on $X$ that may be assumed to have simple normal crossings. The divisor
$\Delta$ is said to be\emph{ klt} (Kawamata Log Terminal) if all
of its coefficients $w_{i}$ satisfy $w_{i}<1$ \cite{ko}. If this
is the case the variety $Y$ is also said to have klt singularities.
In another direction, if $Y$ is a projective algebraic manifold with
positive (but not maximal) Kodaira dimension, then $K_{Y}$ is the
pull-back of $K_{(X,\Delta)}$ for a klt divisor $\Delta$ supported
on the ramification locus of the litaka fibration $Y\rightarrow X.$ 
\end{rem}

As shown in \cite{berm8 comma 5} Theorem \ref{thm:ke intro} can
be generalized to any klt pair $(X,\Delta)$ such that $(X,\Delta)$
has positive Kodaira dimension. 

\subsection{The case $\beta<0$}

When $K_{X}$ is negative, that is, $X$ is a Fano manifold, we replace
the zero-dimensional spaces $H^{0}(X,kK_{X})$ with the spaces $H^{0}(X,-kK_{X})$
of dimension
\[
N_{k}:=\dim H^{0}(X,-kK_{X}),
\]
tending to infinity, as $k\rightarrow\infty.$ We are then forced
to replace the power $2/k$ in formula \ref{eq:canon prob measure intro}
with a negative power $-2/k$ in order to ensure that
\begin{equation}
\left|\det S^{(k)}\right|^{-2/k}\label{eq:density Fano}
\end{equation}
transforms as a density on $X^{N_{k}},$ i.e. defines a global measure
on $X^{N_{k}}.$ However, then the corresponding normalization constant
$Z_{N_{k}}$
\[
Z_{N_{k}}:=\int_{X^{N_{k}}}\left|\det S^{(k)}\right|^{-2/k}
\]
may diverge, since the integrand blows-up along the zero-locus in
$X^{N_{k}}$ of the section $\det S^{(k)}.$ Accordingly, we will
say that a Fano manifold $X$ is\emph{ Gibbs stable at level $k$}
if $Z_{N_{k}}<\infty$ and \emph{Gibbs stable} if it is Gibbs stable
at level $k$ for $k$ sufficiently large. For a Gibbs stable Fano
manifold $X$ we set 
\begin{equation}
\mu^{(N_{k})}:=\frac{1}{Z_{N_{k}}}\left|\det S^{(k)}\right|^{-2/k},\label{eq:canon prob measure Fano intro}
\end{equation}
 which defines a canonical symmetric probability measure on $X^{N_{k}},$
i.e. a canonical random point process on $X$ with $N_{k}$ points.
We thus arrive at the following probabilistic analog of the Yau-Tian-Donaldson
conjecture
\begin{conjecture}
\label{conj:Fano with triv autom intr}Let $X$ be Fano manifold.
Then 
\begin{itemize}
\item $X$ admits a unique Kähler-Einstein metric $\omega_{KE}$ if and
only if $X$ is  Gibbs stable. 
\item If $X$ is Gibbs stable, the empirical measures $\delta_{N}$ of the
corresponding canonical point processes converge in probability towards
the normalized volume form of $\omega_{KE}.$ 
\end{itemize}
\end{conjecture}

If $X$ is  Gibbs stable then

\begin{equation}
\omega_{k}:=-dd^{c}\log\E(\delta_{N_{k}})=-dd^{c}\log\int_{X^{N_{k}-1}}\left|\det S^{(k)}\right|^{-2/k}\label{eq:def of omega k Fano}
\end{equation}
defines a sequence of canonical positive currents (as follows from
the positivity of direct image bundles in \cite{bern1}; see \cite[Prop 6.5]{berm8 comma 5}).
In analogy with Corollary \ref{cor:gen type conv of correl measu intro}
it seems natural to also conjecture that, if $X$ is  Gibbs stable,
then $\omega_{k}$ converges to a Kähler-Einstein metric on $X.$

It should be stressed that the Gibbs stability of $X$ implies that
the group $\text{Aut \ensuremath{(X)_{0}}}$ is trivial \cite[Prop 6.5]{berm8 comma 5}.
Accordingly, when comparing Conjecture \ref{conj:Fano with triv autom intr}
with the Yau-Tian-Donaldson conjecture one should view  Gibbs stability
as the analog of K-stability. 

In the light of Theorem \ref{thm:LDP beta pos text} it is natural
to pose the following stronger LDP form of the previous conjecture:
\begin{conjecture}
\label{conj:LDP fano intro}Let $X$ be a Fano manifold. Then $X$
admits a unique Kähler-Einstein metric iff the canonical measure $\mu^{(N)}$
is a probability measure for $N$ sufficiently large. Moreover, if
this is the case then the laws of the empirical measures $\delta_{N}$
on $(X^{N},\mu^{(N)})$ satisfy a Large Deviation Principle (LDP)
with speed $N$ and rate functional $F_{-1}(\mu)-C_{\beta},$ where
\[
F_{-1}(\mu)=-E_{\omega_{0}}(\mu)+D_{dV}(\mu),\,\,\,C_{-1}=\inf_{\mathcal{P}(X)}F_{-1}.
\]
\end{conjecture}

To highlight the connection to the LDP in Theorem \ref{thm:LDP beta pos text},
fix a volume form $dV$ on $X$ and denote by $\left\Vert \cdot\right\Vert $
the induced metric on the anti-canonical line bundle $-K_{X}.$ Then
the canonical probability measure \ref{eq:canon prob measure Fano intro}
may be expressed as 
\begin{equation}
\mu_{\beta}^{(N)}:=\frac{1}{Z_{N_{k},\beta}}\left\Vert \det S^{(k)}\right\Vert ^{2\beta/k}dV^{\otimes N},\,\,\,\,Z_{N_{k},\beta}:=\int_{X^{N_{k}}}\left\Vert \det S^{(k)}\right\Vert ^{2\beta/k}dV^{\otimes N},\label{eq:def of Gibbs measure Fano setting intro}
\end{equation}
 for $\beta=-1.$ Since the rate functional of a LDP is automatically
lower semi-continuous (lsc) the validity of the LDP in Conjecture
\ref{conj:LDP fano intro} would imply that the free energy functional
$F_{-1}$ is lsc on $\mathcal{P}(X).$ This is indeed the case. More
precisely, the following result holds, deduced from a combination
of results in \cite{ti,bbegz}: 
\begin{thm}
\label{thm:lsc intro}Let $X$ be a Fano manifold and set $\beta=-1.$
Then the free energy functional $F_{\beta}$ on the space $\mathcal{P}(X)$
of probability measures on $X$ is lower semi-continuous iff $X$
admits a unique Kähler-Einstein metric $\omega_{KE}.$ Moreover, if
$F_{\beta}$ is lsc, then the normalized volume form of $\omega_{KE}$
is the unique minimizer of $F_{\beta}.$
\end{thm}

We note in passing that, just as in the case when $K_{X}>0$ and $\beta=1,$
the free energy functional $F_{\beta}$ at $\beta=-1,$ restricted
to the space of volume forms on $\mathcal{P}(X),$ may be identified
with the Mabuchi functional $\mathcal{M}$ on $c_{1}(-K_{X}).$ 

\subsubsection{Symmetry breaking}

As discussed above, Gibbs stability should be considered as the analog
of K-stability. One is thus naturally lead to ask whether there is
also a notion of ``Gibbs polystability'', taking the action of the
group $\text{Aut \ensuremath{(X)_{0}}}$ into account? This is an
intriguing question that we shall sidestep here, by breaking the $\text{Aut \ensuremath{(X)_{0}}}$
-symmetry as follows. Fixing a volume form $dV$ one can viewed the
probability measure $\mu_{\beta}^{(N)},$ defined by formula \ref{eq:def of Gibbs measure Fano setting intro},
as a deformation of the canonical measure $\mu^{(N)}$ to $\beta<-1.$
Since $-K_{X}>0$ we may pick a volume form $dV$ inducing a metric
on $-K_{X}$ with positive curvature: 
\[
\omega_{0}:=\text{Ric \ensuremath{dV>0}}
\]
One advantage of allowing $\beta>-1$ is that it attenuates the singularities
of the integrand, as further discussed below.
\begin{conjecture}
\label{conj:Fano beta greater than minus 1}Assume that $X$ is a
Fano manifold. If $X$ admits a Kähler-Einstein metric, then for any
given $\beta<1$ we have that $Z_{N_{k},\beta}<\infty$ for $k$ sufficiently
large. Moreover, the empirical measure $\delta_{N}$ on $(X^{N_{k}},\mu_{\beta}^{(N_{k})}),$
convergence in law towards a volume form $\mu_{\beta}$ and 
\[
\lim_{\beta\rightarrow}\mu_{\beta}=dV_{KE},
\]
 where $dV_{KE}$ is the volume form of a Kähler-Einstein metric on
$X.$ More precisely, for $\beta>-1$
\[
\omega_{\beta}:=\text{Ric}\ensuremath{dV}+\beta^{-1}dd^{c}\log\frac{\mu_{\beta}}{dV}
\]
is the unique Kähler metric solving 
\end{conjecture}

\begin{equation}
\text{Ric \ensuremath{\omega_{\beta}=-\beta\omega_{\beta}+(1}}+\beta)\text{Ric \ensuremath{dV}},\label{eq:aubin intro}
\end{equation}
One motivation for this conjecture is that it holds for $\beta>0.$
Indeed, according to Theorem \ref{thm:LDP intro} the result holds
in the stronger sense of large deviations. As a consequence, for $\beta>0,$the
LDP also implies that
\begin{equation}
-\lim_{N\rightarrow\infty}\frac{1}{N}\log Z_{N,\beta}=\inf_{\mathcal{P}(X)}F_{\beta},\label{eq:asympt of log Z intro}
\end{equation}
 if the fixed bases in $H^{0}(X,-kK_{X})$ is taken to be orthonormal
with respect to the scalar product induced by $dV.$
\begin{rem}
Incidentally, the equation \ref{eq:aubin intro} coincides with the
one introduced by Aubin's in his continuity method for solving the
Kähler-Einstein equation at $\beta=-1.$ \cite{au2} In Aubin's notation
the time-parameter corresponds to $-\beta\in[0,1]$. The uniqueness
of solutions for $\beta>-1$ was established in \cite{b-m}. Moreover,
it was also shown in \cite{b-m} that if $X$ admits a Kähler-Einstein
metric, then $\omega_{\beta}$ exists for any $\beta>-1$ and converges,
as $\beta\rightarrow-1,$ towards a particular Kähler-Einstein metric,
singled out by $dV.$ 
\end{rem}

From a statistical mechanics point of view it may, at a first glance,
seem rather odd to consider the case when $\beta<0,$ since it would
correspond to a negative (absolute!) temperature. But the notion of
negative temperature states does make sense physically (see the discussion
in \cite[Remark 8.1]{berm10b}). Moreover, from an equivalent point
of view we may as well consider the case of unit temperature and instead
replace the energy particle $E^{(N)}$ with the rescaled energy $\beta E^{(N)}$
(thus treating $\beta$ as a coupling constant). For $\beta>0$ this
energy is is \emph{repulsive,} since it tends to $\infty$ as any
two particle positions merge (due to the vanishing of the determinant
$\det S^{(k)}(x_{1},...,x_{N_{k}})$). However, when $\beta$ changes
sign, from positive to negative, the rescaled energy $\beta E^{(N)}$
becomes\emph{ attractiv}e; it tends to $-\infty$ as any two points
merge.\emph{ }Still, it turns out that the system is sufficiently
mildly attractive to allow $Z_{N,\beta}$ to be finite for $\beta>\beta_{0},$
for some negative $\beta_{0},$ ensuring that the Gibbs measure $\mu_{\beta}^{(N)}$
is well-defined. This may be interpreted as a statistical mechanical
type of stability, since it amounts to the existence of the microscopic
$N-$particle equilibrium state. More precisely, there exists $\beta_{0}\in]-\infty,0[$
such that for any $\beta>\beta_{0}$ 
\[
\frac{1}{N}\log Z_{N,\beta}\geq-C_{\beta},\,\,\,
\]
as $N\rightarrow\infty.$ 

\subsubsection{Stability thresholds and uniform Gibbs stability}

The previous discussion motivates the introduction of the following
``microscopic stability thresholds''

\begin{equation}
\gamma_{N_{k}}(X):=\sup\left\{ \gamma:\,Z_{N_{k},-\gamma}<\infty\right\} \label{eq:def of gamma N k}
\end{equation}
 and their limit 
\[
\gamma(X):=\liminf_{k\rightarrow\infty}\gamma_{N_{k}},
\]
 as well as the ``macroscopic stability threshold'' 

\[
\Gamma(X):=\sup_{\beta>0}\left\{ -\beta:\,\,\,\inf_{\mathcal{P}(X)}F_{\beta}>-\infty\right\} 
\]
In the ``thermodynamical limit'', $N\rightarrow\infty$, it is thus
natural to make the following
\begin{conjecture}
\label{conj:gamma is Gamma}Let $X$ be a Fano manifold. Then the
two invariants $\gamma(X)$ of $X$ and $\Gamma(X),$ defined above,
coincide: 
\[
\gamma(X)=\Gamma(X)
\]
\end{conjecture}

\begin{rem}
The threshold $\gamma_{N}(X)$ may be interpreted as the threshold
where the self-attraction of the $N-$particle system can no longer
be compensated by the disorder resulting from the randomness. Similarly,
the threshold $\Gamma(X)$ is the threshold where the macroscopic
tendency to self-attract and form singular states can no longer be
balanced by the regularizing effect of the entropy. 
\end{rem}

Let us call a Fano manifold \emph{uniformly Gibbs stable} if $\gamma(X)>1$.
This should be thought of us an analog of the notion of uniform K-stability.
By the results in \cite{bbj} a Fano manifold is uniformly K-stable
iff $\Gamma(X)>1.$ Hence, the validity of the previous conjecture
would imply the validity of the following one:
\begin{conjecture}
\label{conj:unif Gibbs iff unif K}Let $X$ be a Fano manifold. Then
$X$ is uniformly Gibbs stable $X$ iff $X$ is uniformly K-stable. 
\end{conjecture}

\begin{rem}
Combining \cite{c-d-s} and \cite{bhj,de} reveals that a Fano manifold
is, in fact, uniformly K-stable iff it is K-stable. This leads one
to wonder whether Gibbs stability, is, in fact, equivalent to the
uniform Gibbs stability? Theorem \ref{thm:log curve intro} below
shows that for log Fano curves this is indeed the case.
\end{rem}

As observed in \cite{berm8 comma 5} the notion of Gibbs stability
introduced above can also be given the following purely\emph{ algebro-geometric
}formulation in the spirit of Minimal Model Program (MMP). Let $\mathcal{D}_{k}$
be the effective divisor in $X^{N_{k}}$ cut out by the section $\mbox{det \ensuremath{S^{(k)}.}}$
Geometrically, $\mathcal{D}_{k}$ may be represented as the following
incidence divisor in $X^{N_{k}}:$ 
\[
\mathcal{D}_{k}:=\{(x_{1},...x_{N})\in X^{N_{k}}:\,\exists s\in H^{0}(X,-kK_{X}):\,s(x_{i})=0,\,i=1,..,N_{k}\}
\]
 Gibbs stability at level $k$ amounts to the anti-canonical $\Q-$divisor
$\mathcal{D}_{k}/k$ on $X^{N_{k}}$ having klt singularities (see
Remark \ref{rem:klt}), which means that 
\begin{equation}
\mbox{lct\ensuremath{(X^{N_{k}},\mathcal{D}_{k}/k)>1}}\label{eq:bound on lct}
\end{equation}
for $k>>1,$ where $\mbox{lct\ensuremath{(X^{N_{k}},\mathcal{D}_{k}/k)}}$
denotes  the\emph{ log canonical threshold (lct)} of the $\Q-$divisor
$\mathcal{D}_{k}/k$ on $X^{N_{k}}$ \cite{ko}. Indeed, it follows
from the standard analytic interpretation of the lct as an integrability
threshold that 
\[
\gamma_{k}(X)=\mbox{lct\ensuremath{(X^{N_{k}},\mathcal{D}_{k}/k)}}
\]

Using properties of log canonical and techniques from the MMP one
direction of Conjecture was established in \cite{f-o} (see also \cite{fu}
where K-stability was first shown):
\begin{thm}
\label{thm:(Fujita-Odaka)-Uniform-Gibbs}(Fujita-Odaka) \cite{f-o}Uniform
Gibbs stability implies uniform K-stability
\end{thm}

Let us briefly recall the elegant argument in \cite{f-o}. First,
by \cite[Thm 2.5]{f-o},
\[
\gamma_{k}(X)\leq\delta_{k}(X):=\inf_{D_{k}}\text{lct }(X,D_{k}),
\]
 where the inf is taken over all anti-canonical $\Q-$divisors $D_{k}$
on $X$ of $k-$basis type, i.e. $D_{k}$ is the normalized sum of
the $N_{k}$ zero-divisors on $X$ defined by the members of a given
basis in $H^{0}(X,-K_{X})$. Finally, by \cite[Thm 0.3]{f-o}, if
the invariant $\delta(X)$ defined as the limsup of $\delta_{k}(X)$
satisfies $\delta(X)>1,$ then $X$ is uniformly K-stable, as follows
from the valuative criterion in \cite{li,fu2} (see also \cite{bbj,bl-j,c-r-z}
for related results). 

Combining Theorem \ref{thm:(Fujita-Odaka)-Uniform-Gibbs} with \cite{c-d-s}
or \cite{bbj} shows that uniform Gibbs stability implies the existence
of a unique Kähler-Einstein metric $\omega_{KE}.$ This is in line
with Conjecture \ref{conj:Fano with triv autom intr}. However, the
converse implication is still open, as well as the problem of establishing
the convergence of the corresponding canonical random point processes
towards $dV_{KE},$ when it exists. Here we will focus on the convergence
problem, introducing a variational approach. As shown in \cite{bbj2}
a non-Archimedean analog of this variational approach also has bearings
om the converse of Theorem \ref{thm:(Fujita-Odaka)-Uniform-Gibbs}.

\subsubsection{A variational approach towards proving the convergence in Conjecture
\ref{conj:Fano with triv autom intr}}

As discussed in Section \ref{sec:Towards-the-case of neg} the proof
of the LDP in Theorem \ref{thm:LDP intro} brakes down when $\beta<0.$
To handle this case a variational approach is proposed in Section
\ref{sec:Towards-the-case of neg}. The approach, which is based on
Gibbs variational principle, reveals that it is enough to establish
the asymptotics \ref{eq:asympt of log Z intro} for $\beta=-1:$ 
\begin{equation}
-\lim_{N\rightarrow\infty}\frac{1}{N}\log Z_{N,-1}=\inf_{\mathcal{P}(X)}F_{-1}\label{eq:conv log Z N minus one intro}
\end{equation}
We show that the upper bound does hold, but the lower bound hinges
on a conjectural upper bound on the mean energy of the $N-$particle
Gibbs measures. By making contact with the theory of phase transitions
in statistical mechanics, we also observe that if there exists $\beta_{0}<-1$
and a function $f(\beta)$ on $]\beta_{0},0[$ such that
\[
-\lim_{N\rightarrow\infty}\frac{1}{N}\log Z_{N,\beta}=f(\beta)
\]
 then the convergence \ref{eq:conv log Z N minus one intro} holds
iff $f(\beta)$ is real-analytic. Hence, if this is the case then
the convergence in Conjecture \ref{conj:Fano with triv autom intr}
holds. Moreover, as explained in Section \ref{subsec:Analyticity-and-absence},
the real-analyticity in question can be related to the distribution
of the poles of the Archimedean zeta functions $Z_{N,\beta}.$

\subsubsection{\label{subsec:Log-Fano-manifolds}The case of log Fano curves}

There is only one-dimensional Fano manifold $X$ - the complex projective
line (the Riemann sphere) - and its Kähler-Einstein metrics are all
biholomorphically equivalent to the standard round metric on the two-sphere.
A geometrically more interesting situation appears when introducing
weighted points (conical singularities) on the Riemann sphere. From
the algebro-geometric point of view this fits into the general setting
of\emph{ log Fano manifolds.} A log pair $(X,\Delta),$ consisting
of a complex manifold $X$ and an effective $\Q-$divisor $\Delta,$
is said to be a \emph{log Fano manifold} if its anti-canonical line
bundle is positive, $-(K_{X}+\Delta)>0.$ The corresponding log Kähler-Einstein
equation \ref{eq:log KE eq with beta intro} for $\beta=-1$ thus
reads
\begin{equation}
\mbox{Ric \ensuremath{\omega}}=-\omega+[\Delta].\label{eq:log KE eq intro}
\end{equation}
 To any log Fano manifold $(X,\Delta)$ we may attach a sequence of
canonical probability measures $\mu_{\Delta}^{(N_{k})}$ on $X^{N_{k}}$
by simply replacing the anti-canonical line bundle $-K_{X}$ with
$-K_{(X,\Delta)}$ and setting
\[
\mu_{\Delta}^{(N_{k})}:=\frac{1}{Z_{N_{k}}}\left|\det S^{(k)}(z_{1},...,z_{N})\right|^{-2/k}|s_{\Delta}|^{-2}(z_{1})\cdots|s_{\Delta}|^{-2}(z_{N_{k}}),
\]
which is a well-defined probability measure, as long as the corresponding
normalizing constant is finite,
\[
Z_{N_{k}}:=\int_{X^{N_{k}}}\left|\det S^{(k)}(x_{1},...,x_{N_{k}})\right|^{-2/k}|s_{\Delta}|^{-2}(x_{1})\cdots|s_{\Delta}|^{-2}(x_{N})<\infty
\]
We then say that log Fano manifold $(X,\Delta)$ is  Gibbs stable.
The invariants $\gamma_{k}(X,\Delta)$ and uniform Gibbs stability
of $(X,\Delta)$ can also be defined as before, mutatis mutandis. 

Now, let $(X,\Delta)$ be a log Fano curve $(X,\Delta),$ i.e. $X$
is the complex projective line and 
\[
\Delta=\sum_{i=1}^{m}w_{i}p_{i}
\]
for positive weights $w_{i}$ satisfying $\sum_{i=1}^{m}w_{i}<2.$
In \cite{berm11} it is shown that the conjectures discussed above
hold for any log Fano curve:
\begin{thm}
\label{thm:log curve intro}Let $(X,\Delta)$ be a log Fano curve.
Then the following is equivalent
\begin{itemize}
\item $(X,\Delta)$ is  Gibbs stable
\item $(X,\Delta)$ is uniformly Gibbs stable
\item The following weight condition holds: 
\begin{equation}
w_{i}<\sum_{i\neq j}w_{j},\,\,\,\forall i\label{eq:weight condition intro}
\end{equation}
\item There exists a unique a unique Kähler-Einstein metric $\omega_{KE}$
for $(X,\Delta)$
\end{itemize}
Moreover, if any of the conditions above hold, then the laws of the
random measures $\delta_{N}$ on $(X^{N},\mu_{\Delta}^{(N)})$ satisfy
a Large Deviation Principle (LDP) with speed $N$ and a rate functional
$I$ with a unique minimizer $\omega_{KE}/\int_{X}\omega_{KE}.$ 
\end{thm}

In this logarithmic setting the rate functional $I$ has the property
that
\[
I(\frac{\omega}{\int_{X}\omega})=\mathcal{M}_{(X,\Delta)}(\omega),
\]
 where $\mathcal{M}_{(X,\Delta)}$ denote the Mabuchi functional for
$(X,\Delta),$ which in the general setting of log Fano varieties
was defined in \cite{bbegz} on the space of finite energy currents
in $\omega$ in $-c_{1}(K_{X}+\Delta).$ As explained in \cite{berm11},
the previous theorem is a direct consequence of the LDP in \cite[Thm 1.5]{berm10},
concerning singular pair interactions, which generalize the vortex
model in two-dimensional hydrodynamics in \cite{clmp,k}. A key ingredient
is an a priori estimate on the correlation measures of the processes,
which builds on \cite{clmp,k} and implies the conjectural energy
bound property \ref{eq:upper bound property} in this setting.

The problem of finding constant curvature metrics on Riemann surfaces
with conical singularities has a long history and was first posed
as a competition topic by the Göttingen Mathematical Society in 1890
\cite{s-g}. The weight condition in the previous theorem first appeared
in \cite{tr}, where the existence of $\omega_{KE}$ was established
and the uniqueness was settled in \cite{l-t}. By \cite[Ex. 6.6]{fu2}
the weight condition \ref{eq:weight condition intro} is equivalent
to the uniform K-stability of $(X,\Delta),$ which thus is equivalent
to uniform Gibbs stability in this setting. 

\section{\label{sec:Background}Background}

\subsection{\label{subsec:Complex-geometry}Complex geometry}

We start by recalling some basic complex geometry - for more background
see, for example, the exposition in \cite{berm10b} and the books
\cite{de,g-z}. Let $X$ be an $n-$dimensional compact complex manifold
and denote by $J$ the corresponding complex structure, viewed as
an endomorphism of the real tangent bundle satisfying $J^{2}=-I.$ 

\subsubsection{Kähler forms/metrics}

On a complex manifold $(X,J)$ anti-symmetric two forms $\omega$
and symmetric two tensors $g$ on $TX\otimes TX,$ which are $J-$invariant,
may be identified by setting 
\[
g:=\omega(\cdot,J\cdot)
\]
Such a real two form $\omega$ is said to be \emph{Kähler} if it is
closed, $d\omega=0,$ and the corresponding symmetric tensor $g$
is positive definite (i.e. defines a Riemannian metric)\footnote{A $J-$invariant two form $\omega$ is usually said to be of type
$(1,1)$ since $\omega=\sum_{i,j}\omega_{ij}dz_{i}\wedge d\bar{z}_{j}$
in local holomorphic coordinates.}. Conversely, a Riemannian metric $g$ is said to be \emph{Kähler
}if it arises in this way (in Riemannian terms this means that parallel
transport with respect to $g$ preserves $J).$ By the local $\partial\bar{\partial}-$
lemma a $J-$invariant two form $\omega$ is closed, i.e. $d\omega=0$
if and only if $\omega$ may be locally expressed as $\omega=\frac{i}{2\pi}\partial\bar{\partial}\phi,$
in terms of a local smooth function $\phi$ (called a local potential
for $\omega)$. In real notation this means that 
\[
\omega=dd^{c}\phi,\,\,d^{c}:=-\frac{1}{4\pi}J^{*}d
\]

\begin{rem}
The normalization above ensures that $dd^{c}\log|z|^{2}$ is a probability
measure on $\C.$
\end{rem}

We will denote by $[\omega]\in H^{2}(X,\R)$ the de Rham cohomology
class represented by $\omega.$ If $\omega_{0}$ is a fixed Kähler
form then, according to the global $\partial\bar{\partial}-$ lemma,
any other Kähler metric in $[\omega_{0}]$ may be globally expressed
as 
\[
\omega_{\varphi}:=\omega_{0}+dd^{c}\varphi,\,\,\,\,\varphi\in C^{\infty}(X),
\]
where $\varphi$ is determined by $\omega_{0}$ up to an additive
constant and called a Kähler potential for $\omega_{\varphi}.$ The
space of all Kähler potentials is denoted by
\[
\mathcal{H}(X,\omega):=\left\{ \varphi\in C^{\infty}(X):\,\omega_{\varphi}>0\right\} 
\]
The association $\varphi\mapsto\omega_{\varphi}$ thus allows one
to identify $\mathcal{H}(X,\omega)/\R$ with the space of all Kähler
forms in $[\omega_{0}].$ 

\subsubsection{\label{subsec:Metrics-on-line}Metrics on line bundles and curvature}

Let $L$ be a holomorphic line bundle on $X$ and $\left\Vert \cdot\right\Vert $
a Hermitian metric on $L.$ The normalized curvature two-form of $\left\Vert \cdot\right\Vert $
may be (locally) written as 
\begin{equation}
\omega:=-dd^{c}\phi\log\left\Vert e_{U}\right\Vert ^{2},\,\,\,\phi:=-\log\left\Vert e_{U}\right\Vert ^{2}\label{eq:def of phi as weight}
\end{equation}
 in terms of a given local trivialization of $L,$ i.e. a non-vanishing
holomorphic section $e_{U}$ of $L$ over $U\subset X.$ The local
function $\phi$ is called the \emph{weight} of the metric. The corresponding
cohomology class $[\omega]$ is independent of the metric $\left\Vert \cdot\right\Vert $
on $L$ and coincides with the first Chern class $c_{1}(L)$ in $H^{2}(X,\R)\cap H^{2}(X,\Z)$
(conversely, any such cohomology class is the first Chern class of
line bundle $L).$ A line bundle $L$ is said to be\emph{ positive}
if it admits a metric with positive curvature, i.e. such that the
curvature form $\omega$ is Kähler. Then the pair $(X,L)$ is called
a\emph{ polarized manifold.} Fixing a reference metric $\left\Vert \cdot\right\Vert _{0}$on
$L$ with curvature form $\omega_{0}$ any other metric on $L$ may
be expressed as 
\[
\left\Vert \cdot\right\Vert =\left\Vert \cdot\right\Vert _{0}e^{-\varphi/2},
\]
 for $\varphi\in C^{\infty}(X).$ The curvature form of $\left\Vert \cdot\right\Vert $
is thus given by $\omega_{\varphi}$ which is positive iff $\varphi\in\mathcal{H}(X,\omega_{0}).$
Note that the definitions are made so that 
\[
\varphi=\phi-\phi_{0}
\]
 in terms of the local weights $\phi$ and $\phi_{0}$ of the metrics
$\left\Vert \cdot\right\Vert $ and $\left\Vert \cdot\right\Vert _{0}.$
Any positive line bundle $L$ is\emph{ big}, i.e. there exists $V>0$
(called the volume of $L)$ such that
\[
N_{k}:=\dim H^{0}(X,kL)=Vk^{n}+o(k^{n}),\,\,\,k\rightarrow\infty
\]
This follows, for example, from the Kodaira embedding theorem (recalled
below) and the volume $V$ may be expressed as
\[
V:=\int_{X}\omega_{0}^{n}
\]
Given a metric $\left\Vert \cdot\right\Vert $ on $L$ we will use
the same notation $\left\Vert \cdot\right\Vert $ for the induced
metric on the tensor products of $L$ over $X,$ obtained by imposing
that $\left\Vert \cdot\right\Vert $ be multiplicative. In particular,
if $\phi$ is a local weight for $\left\Vert \cdot\right\Vert $ (defined
with respect to the local trivialization $e_{U})$ then $k\phi$ is
a local weight for the $k$ th tensor product of $L$ (defined with
respect to the local trivialization $e_{U}^{\otimes k}).$ This motivates
using the additive notation $kL$ for tensor products. More generally,
we will use the same notation $\left\Vert \cdot\right\Vert $ for
the induced metrics on the line bundles $(kL)^{\boxtimes N}$ over
the $N-$fold products $X^{N}.$ 

\subsubsection{\label{subsec:Projectiv-algebraic-embeddings}Algebraic embeddings
of polarized manifolds}

Recall that the the $m-$dimensional complex projective space $\P_{\C}^{m}$
is defined by 
\[
\P_{\C}^{m}(:=\C^{m+1}/\C^{*}
\]
Denote, as usual, by $\mathcal{O}(1)$ the hyperplane line bundle
over $\P_{\C}^{m},$ i.e. the dual of the tautological line bundle
$\C^{m+1}\rightarrow\P_{\C}^{m}.$ The space $H^{0}(X,k\mathcal{O}(1))$
may be identified with the space of all homogeneous polynomials on
$\C^{m+1}$ of degree $k.$ The line bundle $\mathcal{O}(1)\rightarrow\P_{\C}^{m}$
comes with a positively curved metric, namely the \emph{Fubini-Study
metric} induced by the Euclidean norm on $\C^{m+1}$ (see \cite[Section 3.7]{berm10b}
for more background). Hence, $(X,L)$ is a polarized manifold, in
the sense of the previous section. More generally, if $X$ is a complex
submanifold of $\P_{\C}^{m}$ (which, by Chow's theorem, equivalently
means that $X$ is a algebraic) then $(X,\mathcal{O}_{X}(1))$ is
a polarized manifold, where $\mathcal{O}(1)_{X}$ denotes the restriction
of $\mathcal{O}(1)\rightarrow\P^{m}$ to $X.$ Indeed, the restriction
to $\mathcal{O}_{X}(1)$ of the Fubini-Study metric on $\mathcal{O}(1)$
is positively curved. Conversely, by the Kodaira embedding theorem,
if $(X,L)$ is a polarized manifold, then after perhaps replacing
$L$ by a large tensor power, $X$ may be holomorphically embedded
in a projective space $\P^{m}$ in such a way that $L$ gets identified
with $\mathcal{O}(1)_{X}:$ 
\begin{equation}
X\rightarrow\P(H^{0}(X,L)^{*}),\,\,\,\,x\mapsto[s_{0}(x):...:s_{m}(x)]\in\P^{m}\label{eq:Kod emb}
\end{equation}
 where $x$ is mapped to the evaluation functional at $x$ and $s_{0},...,s_{m}$
denotes a fixed basis in $\P(H^{0}(X,L).$ Thus, for $k$ sufficently
large, $H^{0}(X,kL)$ identifies with the restriction to $X$ of the
space of all $k-$homogeneous polynomials over $\P^{m}.$ By Chow's
theorem the embedding of $X$ is an algebraic submanifold and hence
a line bundle is positive iff it is \emph{ample,} in the algebro-geometric
sense. 

\subsubsection{The canonical line bundle and Ricci curvature}

When $L$ is the canonical line bundle $K_{X},$ i.e. the top exterior
power of the holomorphic cotangent bundle of $X:$ 
\[
K_{X}:=\det(T^{*}X)
\]
any volume form $dV$ on $X$ induces a smooth metric $\left\Vert \cdot\right\Vert _{dV}$
on $K_{X},$ by locally setting 
\begin{equation}
\left\Vert dz\right\Vert _{dV}^{2}:=c_{n}dz\wedge d\bar{z}/dV,\label{eq:local def of metric on KX}
\end{equation}
 where $dz:=dz_{1}\wedge\cdots\wedge dz_{n}$ is the local holomorphic
trivialization of $K_{X}$ induced by a choice of local holomorphic
coordinates and $c_{n}dz\wedge d\bar{z}$ is a short hand for the
local Euclidean volume form $\frac{i}{2}dz_{1}\wedge d\bar{z}_{1}\wedge\cdots\wedge\frac{i}{2}dz_{n}\wedge d\bar{z}_{n}.$
When $dV$ is the volume form of a given Kähler metric $\omega$ on
$X,$ i.e. $dV=\omega^{n}/n!,$ then its curvature form may be identified
with minus the Ricci curvature of $\omega,$ i.e. 
\begin{equation}
\mbox{\ensuremath{\mbox{Ric}\,}\ensuremath{\omega=}}-dd^{c}\log\frac{dV}{c_{n}dz\wedge d\bar{z}}.\label{eq:def of ricci curv}
\end{equation}
By a slight abuse of notation we will also write $\mbox{Ric \ensuremath{(dV)} }$for
the right hand side in formula \ref{eq:def of ricci curv}. 

\subsubsection{\label{subsec:Twisted-K=0000E4hler-Einstein-metrics}Twisted Kähler-Einstein
metrics }

A Kähler metric $\omega_{\beta}$ is said to be a \emph{twisted Kähler-Einstein
metric} if it satisfies the twisted Kähler-Einstein equation 
\begin{equation}
\mbox{\ensuremath{\mbox{Ric}}\ensuremath{\omega}}=-\beta\omega+\theta,\label{eq:tw ke eq in text}
\end{equation}
where the form $\theta$ is called the\emph{ twisting form. }Since
$\omega_{\beta}$ is Kähler the form $\eta$ is necessarily closed
and $J-$invariant, i.e. of type $(1,1).$ The corresponding equation
at the level of cohomology classes is

\begin{equation}
[\omega]=\frac{1}{\beta}\left([\theta]+c_{1}(K_{X})\right)\label{eq:cohom eq}
\end{equation}

\begin{rem}
There is no loss of generality if one assumes that that $|\beta|=1$
(by replacing $\omega$ with $|\beta|\omega),$ but allow general
$\beta$ makes the connection to the statistical mechanical framework
more apparent. Moreover, allowing $\beta$ to vary continuously is
important for Aubin's method of continuity \cite{au2}.
\end{rem}

To a given pair $(dV,\omega_{0})$ consisting of a volume form $dV$
and a Kähler form $\omega_{0}$ on $X$ we associate, for any parameter
$\beta\in\R,$ the twisting form 

\begin{equation}
\theta:=\beta\omega_{0}+\ensuremath{\mbox{Ric}}\ensuremath{dV}.\label{eq:def of eta}
\end{equation}
This association is invariant under $(dV,\omega_{0})\longrightarrow(e^{\beta u}dV,\omega_{0}+ddu)$
for any $u\in C^{\infty}(X).$

The following lemma follows directly from the expression \ref{eq:def of ricci curv}
for the Ricci curvature of a Kähler metric:
\begin{lem}
\label{lem:twisted}Let $X$ be a compact complex manifold endowed
with a Kähler form $\omega_{0}$ and volume form $dV.$ Then a Kähler
form $\omega_{\beta}\in[\omega_{0}]$ solves the corresponding twisted
Kähler-Einstein equation \ref{eq:tw ke eq in text} iff $\omega_{\beta}:=\omega_{0}+dd^{c}\varphi_{\beta}$
for a unique $\varphi_{\beta}\in\mathcal{H}(X,\omega_{0})$ solving
the PDE
\begin{equation}
\omega_{\varphi}^{n}=e^{\beta\varphi}dV\label{eq:ma eq with beta in text}
\end{equation}
\end{lem}

The Aubin-Yau theorem may now be formulated as follows:
\begin{thm}
\label{thm:(Aubin-Yau)--Assume}(Aubin-Yau) \cite{au,y} Assume given
a compact complex manifold $X,$ endowed with a Kähler form $\omega_{0}$
and a volume form $dV.$ Then the PDE \ref{eq:ma eq with beta in text}
admits, for any positive number $\beta\in]0,\infty[,$ a unique solution
$\varphi_{\beta}\in\mathcal{H}(X,\omega_{0}).$ Equivalently, given
a closed $(1,1)-$form $\theta$ such that the cohomology class $[\theta]+c_{1}(K_{X})$
is positive (i.e. contains a Kähler form) there exists a unique Kähler
metric $\omega_{\beta}$ in $\left([\theta]+c_{1}(K_{X})\right)/\beta$
solving the twisted Kähler-Einstein equation \ref{eq:tw ke eq in text}.
\end{thm}

\begin{proof}
We recall that the uniqueness follows from the maximum principle,
which also yields a priori $C^{0}(X)-$estimates. As for the existence
it is shown using a method of continuity, based on the Aubin-Yau Laplacian
estimates. 
\end{proof}
\begin{example}
A complex manifold $X$ admits a Kähler-Einstein metric with negative
Ricci curvature iff $K_{X}$ is positive (and the metric is unique).
Indeed, if $K_{X}$ is positive then, by the very definition of positivity,
we can take $\omega_{0}:=-\ensuremath{\mbox{Ric}}\ensuremath{dV}$
for some volume form on $X,$ ensuring that $\theta=0$ above, with
$\beta=1$ (and the converse is trivial). 
\end{example}

\subsubsection{\label{subsec:The-Fano-setting}The Fano setting}

Let $X$ be a Fano manifold, i.e. $-K_{X}>0$ and fix a volume form
$dV$ on $X$ with the property that $\text{Ric}dV>0.$ The ``Fano
setting'' will refer to the special situation when the geometric
data is of the form $(dV,\omega_{0})$ with 
\[
\omega_{0}:=\text{Ric}\,dV
\]
For any given $\beta\neq0$ the corresponding twisted Kähler-Einstein
equation \ref{eq:tw ke eq in text} is then of the form 
\begin{equation}
\text{Ric \ensuremath{\omega_{\beta}=-\beta\omega_{\beta}+(1}}+\beta)\text{\ensuremath{\text{Ric}dV}},\label{eq:Aubins eq text}
\end{equation}
For $\beta=-1$ this is precisely the Kähler-Einstein equation \ref{eq:tw ke eq in text}
for a metric with positive Ricci curvature, while for a general $\beta\in[-1,0[$
this is Aubin's continuity equation with ``time-parameter'' 
\[
\gamma:=-\beta\in]0,1]
\]

The fixed volume form $dV$ corresponds to a metric $\left\Vert \cdot\right\Vert $
on $-K_{X},$ which is the dual of the metric induced by $dV$ on
$K_{X}$ (formula \ref{eq:local def of metric on KX}). In other words,
\[
\left\Vert \frac{\partial}{\partial z_{1}}\wedge\cdots\wedge\frac{\partial}{\partial z_{n}}\right\Vert ^{2}:=dV/c_{n}dz\wedge d\bar{z}
\]
Denoting by $\phi_{0}$ the local weight of the metric we can thus
locally express 
\[
dV=c_{n}e^{-\phi_{0}}dz\wedge d\bar{z},\,\,\,\,dd^{c}\phi_{0}=\text{Ric \ensuremath{dV}}
\]
Accordingly, the complex Monge-Ampère equation which is equivalent
to Aubin's equation (Lemma \ref{lem:twisted}) may be locally expressed
as 

\[
(dd^{c}\phi_{\beta})^{n}=c_{n}e^{-\left(\gamma\phi+(1-\gamma)\phi_{0}\right)}dz\wedge d\bar{z},
\]
 where $\phi_{\beta}:=\phi_{0}+\varphi_{\beta}.$ 
\begin{thm}
(Bando-Mabuchi \cite{b-m}) \label{thm:B-M}For $\beta>-1$ the equation
\ref{eq:Aubins eq text} admits at most one solution and for $\beta=-1$
a solution is uniquely determined modulo the action of the group $\text{Aut \ensuremath{(X)_{0}}}.$ 
\end{thm}

\begin{proof}
The proof is based on Aubin's method of continuity \cite{au2}, deforming
from $\beta=0$ to $\beta=-1,$using the uniqueness at $\beta=0$
of the Calabi equation and the uniqueness property for the linearized
equations for $\beta>-1,$ which follows from the Bochner-Kodaira-Nakano
inequality (which also holds for $\beta=-1$ when the group $\text{Aut \ensuremath{(X)_{0}}}$
is trivial). An alternative proof is given in \cite{bern2} which
generalizes to the log Fano setting, as discussed in Section \ref{subsec:The-log-Fano}.
\end{proof}

\subsubsection{\label{subsec:The-log-Fano}The log Fano setting}

Let $X$ be a complex manifold and $\Delta$ a $\Q-$divisor on $X,$
i.e. a formal sum
\[
\Delta=\sum_{i=1}^{m}w_{i}\Delta_{i}
\]
of irreducible subvarieties $\Delta_{i}$ of codimension one in $X,$
with coefficients $w_{i}\in\Q.$ The pair $(X,\Delta)$ is called
a log pair \cite{ko}  and $(X,\Delta)$ is called a \emph{log Fano
manifold} if the anti-canonical line bundle $L$ of the pair $(X,\Delta)$
is positive
\[
L:=-(K_{X}+\Delta)>0,
\]
 where we have identified $\Delta$ with the $\Q-$line bundle $L_{\Delta}$
defined by $\Delta,$ which admits a (multivalued) meromorphic section
$s_{\Delta}$ with multiplicities $w_{i}$ along $\Delta_{i}$ (this
means that $s_{\Delta}^{\otimes l}$ is a well-defined meromorphic
section of the line-bundle $lL_{\Delta}$ for any sufficiently divisible
positive integer $l).$ We fix such a section $s_{\Delta}$ (which
is uniquely determined up to a non-zero multiplicative constant).

Fix a metric $\left\Vert \cdot\right\Vert $ on $L$ with positive
curvature form $\omega_{0}.$ Using that $s_{\Delta}$ defines a canonical
(multi-valued) trivialization of the line bundle $L_{\Delta}$ on
$X-\Delta$ we then get a measure on $X$ defined as in the previous
section on $X-\Delta$ and then extended by zero to all of $X.$ In
other words, locally on $X$ we can express 
\begin{equation}
\mu_{0}=c_{n}e^{-\phi_{0}+\phi_{\Delta}}dz\wedge d\bar{z},\,\,\,,\label{eq:local expres for mu zero}
\end{equation}
 where $\phi_{0}$ is the weight of the metric $\left\Vert \cdot\right\Vert $
on $L$ and $\phi_{\Delta}$ is the weight of the singular metric
on $L_{\Delta}$ induced by $s_{\Delta}.$ This formula shows that,
in the log smooth case, $\mu_{0}$ has an $L_{loc}^{p}-$density for
some $p>1$ iff all coefficients of $\Delta$ are in $]-\infty,1[.$
In general, if $\mu_{0}$ has a $L_{loc}^{p}-$density for some $p>1$
then the log pair $(X,\Delta$) is said to be\emph{ (sub) klt} (which,
in algebraic terms means that the log canonical threshold of $(X,\Delta$)
is $>1$ \cite{ko,bbegz}). Moreover, formula \ref{eq:local expres for mu zero}
reveals that

\[
\text{Ric \ensuremath{\mu_{0}-[\Delta]=\omega_{0},} }
\]
where $[\Delta]$ denotes the current of integration defined by $\Delta.$
The ``log Fano setting'' will refer to the situation when the geometric
data is of the form $(\mu_{0},\omega_{0})$ as above. In this setting,
$\varphi_{\beta}\in\mathcal{E}^{1}(X)$ satisfies\emph{ }
\[
(\omega_{0}+dd^{c}\varphi_{\beta})^{n}=e^{\beta\varphi_{\beta}}\mu_{0}
\]
iff the $(1,1)$ current $\omega_{\beta}:=\omega_{0}+dd^{c}\varphi_{\beta}$
satisfies 
\[
\text{Ric \ensuremath{\omega_{\beta}-[\Delta]=-\beta\omega_{\beta}+(1}}+\beta)\text{\ensuremath{\text{Ric}dV}}
\]
in a weak sense \cite{bbegz}. In particular, for $\beta=-1$ this
is the log Kähler-Einstein equation \ref{eq:log KE eq intro} for
the log pair $(X,\Delta).$ 

\subsection{\label{subsec:Pluripotential-theory}Pluripotential theory }

Recall first that a function $\phi(z)$ in $\C^{n},$ taking values
in $[-\infty,\infty[$ is said to be \emph{plurisubharmonic} (\emph{psh,}
for short) if it is subharmonic along all complex lines. Equivalently,
this means that $\phi$ can be written as a decreasing limit of smooth
functions $\phi_{j}$ such that $dd^{c}\phi_{j}$ is strictly positive,
i.e defines a Kähler form. For $\phi$ psh $dd^{c}\phi$ defines a
positive $(1,1)-$current. Coming back to the global setting of a
compact complex manifold $X$ endowed with a holomorphic line bundle
$L$ a singular metric $\left\Vert \cdot\right\Vert $ om $L$ (taking
values in $[0,\infty[$) is said to be \emph{psh}, i.e. $\phi\in\text{PSH \ensuremath{(L),}}$
if the corresponding local functions $\phi$ (formula \ref{eq:def of phi as weight})
are psh. The curvature of a psh metric on $L$ thus defines a global
positive $(1,1)-$current on $X.$ When $L$ is positive, i.e. admits
a smooth reference metric $\left\Vert \cdot\right\Vert _{0}$ whose
curvature form $\omega_{0}$ defines a Kähler form on $X,$ we can
identify $PSH(X,L)$ with the space $PSH(X,\omega_{0})$ of all $\omega_{0}-$psh
functions $\varphi$ on $X,$ i.e. 
\[
PSH(X,L)\longleftrightarrow PSH(X,\omega_{0}):=\left\{ \varphi\in L^{1}(X):\,\omega_{\varphi}\geq0\right\} ,
\]
 where $\varphi$ is assumed to be a strongly usc function (in order
to make the representation in $L^{1}(X)$ unique). By Demailly's global
approximation result a function $\varphi$ is in $PSH(X,\omega_{0})$
iff it can be written as a decreasing limit $\varphi_{j}$ in $\mathcal{H}(X,\omega_{0}).$ 
\begin{example}
Any element $s_{k}\in H^{0}(X,kL)$ induces a singular psh metric
$\phi$ on $L$ corresponding to $\varphi\in H^{0}(X,kL)$ defined
by $\varphi(x):=k^{-1}\log\left\Vert s_{k}\right\Vert ^{2}(x).$ Thus
the curvature of $\phi$ is the current of integration along the subvariety
defined by the zero-locus of $s,$ including multiplicities. 
\end{example}

\subsubsection{The complex Monge-Ampère operator and the pluricomplex energy $E(\mu)$}

For $\varphi\in\mathcal{H}(X,\omega_{0})$ the complex Monge-Ampère
measure $MA(\varphi)$ is the probability measure on $X$ defined
by the normalized volume form of the Kähler form $\omega_{\varphi}:$
\begin{equation}
MA(\varphi):=\omega_{\varphi}^{n}/V\,\,\,\,V:=\int_{X}\omega_{0}^{n}\label{eq:def of ma meas text}
\end{equation}
using exterior products. The map $\varphi\mapsto MA(\varphi),$ viewed
as a one-form on the convex space $\mathcal{H}(X,\omega_{0})$ is
exact and thus admits a primitive, denoted by $\mathcal{E}.$ In other
words, $\mathcal{E}$ is the functional on $\mathcal{H}(X,\omega_{0})$
defined by the property that 
\begin{equation}
d\mathcal{E}_{|\varphi}=MA(\varphi),\label{eq:def of energyfunc as primitive}
\end{equation}
in the sense that 
\[
\frac{d\mathcal{E}\left(\varphi_{0}+t(\varphi_{1}-\varphi_{0})\right)}{dt}_{|t=0}=\int_{X}MA(\varphi)(\varphi_{1}-\varphi_{0})
\]
and the normalization condition $\mathcal{E}(0)=0.$ Occasionally,
we will write $\mathcal{E}_{\omega_{0}}$ to indicate the dependence
of $\mathcal{E}$ on the normalization. Integrating the defining relation
\ref{eq:def of energyfunc as primitive} along the affine line segment
in\textbf{ }$\mathcal{H}(X,\omega_{0})$ between $\varphi_{0}:=0$
and $\varphi_{1}:=\varphi$ reveals that 
\begin{equation}
\mathcal{E}(\varphi)=\frac{1}{(n+1)V}\int_{X}\varphi\sum_{j=0}^{n}\omega_{\varphi}^{j}\wedge\omega_{\varphi}^{n-j}\label{eq:beauti E as integral}
\end{equation}
We will also denote by $\mathcal{E}_{\omega_{0}}$ the smallest upper
semi-continuous extension of $\mathcal{E}_{\omega_{0}}$ to all of
$PSH(X,\omega_{0})$ and write 
\[
\mathcal{E}^{1}(X):=\left\{ \varphi\in PSH(X,\omega):\,\,\mathcal{E}_{\omega_{0}}(\varphi)>-\infty\right\} ,
\]
 which is called the space of all functions on $X$ with \emph{finite
energy. }The differential property \ref{eq:def of energyfunc as primitive}
still holds on the whole space $\mathcal{E}^{1}(X)$ if the Monge-Ampère
measure $MA(\varphi)$ is defined in terms of non-polar products of
positive currents (see \cite{bbgz}).

Now, following \cite{bbgz} the\emph{ pluricomplex energy} $E_{\omega_{0}}(\mu)$
of a probability measure $\mu$ is defined by
\begin{equation}
E_{\omega_{0}}(\mu):=\sup_{\varphi\in\mathcal{E}^{1}(X)}\mathcal{E}_{\omega_{0}}(\varphi)-\left\langle \varphi,\mu\right\rangle \in]-\infty,\infty]\label{eq:def of e as sup}
\end{equation}
 on $\mathcal{P}(X).$ The functional $E_{\omega_{0}}$ thus defined
is lsc on $\mathcal{P}(X)$ (since it is the the sup of lsc (affine)
functionals, using that $\varphi$ is usc). 
\begin{example}
\label{exa:E on Riemann}In the classical case $n=1,$ i.e. when $X$
is a Riemann surface, 
\[
E_{\omega_{0}}(\mu)=\frac{1}{2}\int_{X}G_{0}(x,y)\mu\otimes\mu,
\]
 where $G_{0}(x,y)$ is the corresponding Green function, i.e. the
symmetric lsc function in $L^{1}(X\times X)$ determined by $dd^{c}G_{0}(\cdot,y)=\omega_{0}-\delta_{y}$
and $\int_{X}G(x,y)\omega(x)=0.$ In electrostatic terms this means
that $E_{\omega_{0}}(\mu)$ is the classical Coulomb energy of a positive
charge distribution $\mu$ on $X$ in the ``neutralizing back-ground
charge $\omega_{0}"$ (compare \cite{berm 1 komma 5}). 
\end{example}

As shown in \cite{bbgz} the Monge-Ampère operator yields a bijection
\begin{equation}
\,\,\varphi\mapsto MA(\varphi)\label{eq:bijection finite energy}
\end{equation}
 between the space $\mathcal{E}^{1}(X)/\R$ and the space of all probability
measure of finite energy. The proof uses a direct variational approach
where the\emph{ potential} $\varphi_{\mu}$ of a measure $\mu$ of
finite energy, i.e. the solution,

\[
MA(\varphi_{\mu})=\mu
\]
is obtained as the element $\mathcal{E}^{1}(X)$ realizing the sup
defining $E_{\omega_{0}}(\mu)$ (formula \ref{eq:def of e as sup}).
In particular, 
\begin{equation}
E_{\omega_{0}}(\mu)=\mathcal{E}_{\omega_{0}}(\varphi_{\mu})-\left\langle \varphi_{\mu},\mu\right\rangle \label{eq:expres for E mu in terms of potential}
\end{equation}

\begin{rem}
In the case when $\mu$ is a volume form the existence of a smooth
solution $\varphi_{\mu}$ was first shown by Yau \cite{y} in the
solution of the Calabi conjecture (the uniqueness of such solutions
is due to Calabi).

Inverting the relation \ref{eq:def of energyfunc as primitive} reveals
that the differential of $dE$ at a volume form in $\mathcal{P}(X)$
is given by
\begin{equation}
dE_{|\mu}=-\varphi_{\mu}\in\mathcal{H}(X)/\R.\label{eq:differ of E}
\end{equation}
 More generally, by \cite[Prop 2.7]{berm6} this formula holds on
all of $\mathcal{P}(X)$ in the sense of sub-gradients. That is to
say that for any two elements $\mu_{0}$ and $\mu_{1}$ in $\mathcal{P}(X)$
of finite energy 
\end{rem}

\begin{equation}
E(\mu_{1})\geq E(\mu_{0})-\left\langle \varphi_{\mu},\mu_{1}-\mu_{0}\right\rangle \label{eq:sub-grad}
\end{equation}

\subsubsection{\label{subsec:The-psh-projection-}The psh-projection $P_{\theta}$
and the Legendre-Fenchel transform of $E$ }

The ``psh-projection'' is the operator $P_{\omega_{0}}$ from $C^{0}(X)$
to $PSH(X,\omega_{0})$ defined as the following envelope:

\begin{equation}
(P_{\omega_{0}}u)(x):=\sup_{\varphi\in PSH(X,\omega_{0})}\{\varphi(x):\,\,\,\varphi\leq u\}\label{eq:def of proj operator in khler case}
\end{equation}
Using the operator $P$ the pluricomplex energy $E_{\omega_{0}}$
may be realized as a Legendre-Fenchel transform (see the general definition
\ref{eq:def of Leg-F trans} below):
\begin{prop}
\label{prop:pluri energy as legendre}The pluricomplex energy $E_{\omega_{0}},$
extended by $\infty$ to the space $\mathcal{M}(X)$ of all signed
measures on $X$, is the \emph{Legendre-Fenchel transform} of the
convex functional 
\begin{equation}
f(u):=-\mathcal{E}_{\omega_{0}}(P_{\omega_{0}}-u)\label{eq:def of f u in Gateaux}
\end{equation}
 Equivalently, this means that
\begin{equation}
E_{\omega_{0}}(\mu)=\sup_{u\in C^{0}(X)}\mathcal{E}_{\omega_{0}}(P_{\omega_{0}}u)-\left\langle u,\mu\right\rangle .\label{eq:E as Leg}
\end{equation}
\end{prop}

\begin{proof}
Since we have assumed that $L>0$ this follows readily from monotonicity
arguments and the fact that any $\varphi\in PSH(X,\omega_{0})$ is
the decreasing limit of functions in $PSH(X,\omega_{0})\Cap C^{0}(X).$
But as shown in \cite{berm 1 komma 5} the result, in fact, holds
more generally for any big cohomology class. 
\end{proof}
We also recall the following differentiability result from \cite{b-b},
which plays a key rule in the variational approach to complex Monge-Ampère
equations in \cite{bbgz}:
\begin{thm}
\label{thm:diff}The convex functional $f(u)$ on $C^{0}(X),$ defined
by formula \ref{eq:def of f u in Gateaux}, is Gateaux differentiable
and its differential at $u\in C^{0}(X)$ is given by
\[
(df)(u)=MA(P_{\omega_{0}}u)
\]
\end{thm}

\subsection{\label{subsec:Probabilistic-setup}Probability}

We recall some basic notions of probability theory (covered by any
standard textbook; see in particular \cite{d-z} for an introduction
to large deviation techniques). A \emph{probability space }is a space
$\Omega$ equipped with a probability measure $p$ and a collection
$\mathcal{F}$ of $p-$measurable subsets $\mathcal{B}\subset\Omega.$
For our purposes it will be enough to consider the case when $\Omega$
is a compact topological space and then we will always take $\mathcal{F}$
to be the collection of all Borel subsets of $\Omega.$ In general,
the space $\Omega$ is called the\emph{ sample space} and a measurable
subset $\mathcal{B}\subset\Omega$ is called an \emph{event} with
\[
\mbox{Prob}\mathcal{B}:=p(\mathcal{B}),
\]
interpreted as the probability of observing the event $\mathcal{B}$
when sampling from $(\Omega,p).$ A measurable function $Y:\,\Omega\rightarrow\mathcal{Y}$
on a probability space $(\Omega,p)$ is called a\emph{ random element
with values in $\mathcal{Y}$ }and its\emph{ law} $\Gamma$ is the
probability measure on $\mathcal{Y}$ defined by the push-forward
measure
\[
\Gamma:=Y_{*}p
\]
 (the law of $Y$ is often also called the distribution of $Y$).
A sequence of random elements $Y_{N}:\,\Omega_{N}\rightarrow\mathcal{Y}$
of probability spaces $(\Omega_{N},p_{N}),$ taking values in the
same topological space $\mathcal{Y}$ are said to\emph{ convergence
in law towards a deterministic element $y$ in $\mathcal{Y}$ }if
the corresponding laws $\Gamma_{N}$ on $\mathcal{Y}$ converge  to
a Dirac mass at $y:$ 
\[
\lim_{N\rightarrow\infty}\Gamma_{N}=\delta_{y}
\]
 in the weak topology. In the present setting $\mathcal{Y}$ will
always be a separable metric space with metric $d$ and then $Y_{N}$
converge in law towards the deterministic element $y$ iff $Y_{N}$
\emph{converge in probability} towards $y,$ i.e. for any fixed $\epsilon>0$
\[
\lim_{N\rightarrow\infty}p_{N}\{d(Y_{N},y)>\epsilon\}=0.
\]

\begin{rem}
The\emph{ expectation} of a random variable $Y$ it defined by 
\[
\E(Y):=\int_{\Omega}Yp
\]
 (aka the\emph{ sample mean} of $Y)$ which defines an element in
$\mathcal{Y}.$ The statement that $Y_{N}$ converges in law towards
a deterministic element $y$ equivalently means that $\E(Y_{N})\rightarrow y$
and that $Y_{N}$ satisfies the (weak) law of large numbers, i.e.
the probability that $Y_{N}$ deviates from its mean tends to zero,
as $N\rightarrow\infty.$
\end{rem}

A\emph{ random point process} \emph{with $N$ particles }on a space
$X$ is, by definition, a probability measure $\mu^{(N)}$ on the
$N-$fold product $X^{N}$ (the $N-$particle space) which is symmetric,
i.e. invariant under action of the group $S_{N}$ of permutations
of the factors of $X^{N}.$The\emph{ empirical measure} of a given
random point process is the following random measure 
\begin{equation}
\delta_{N}:\,\,X^{N}/S_{N}\rightarrow\mathcal{P}(X),\,\,\,(x_{1},\ldots,x_{N})\mapsto\delta_{N}(x_{1},\ldots,x_{N}):=\frac{1}{N}\sum_{i=1}^{N}\delta_{x_{i}}\label{eq:empirical measure text}
\end{equation}
on $(X^{N},\mu^{(N)}).$ The law of $\delta_{N}$ thus defines a probability
measure on the space $\mathcal{P}(X),$ that we shall denote by $\Gamma_{N}:$
\begin{equation}
\Gamma_{N}:=(\delta_{N})_{*}\mu^{(N)}\label{eq:def of Gamma N}
\end{equation}

\begin{rem}
\label{rem:The-point-correlation}\emph{The $j-$point correlation
measure $(\mu^{(N)})_{j}$ }of the $N-$particle random point process
is the probability measure on $X^{j}$ defined as the push-forward
to $X^{j}$ of $\mu^{(N)}$ under projection $X^{N}\rightarrow X^{j},$
where $(x_{1},...,x_{N})\mapsto(x_{i_{1}},...,x_{i_{j}})$ for any
choice of $j$ different indices $i_{1},...,i_{j}.$ In particular,
by symmetry,
\begin{equation}
\E(\delta_{N})=(\mu^{(N)})_{1}\label{eq:expect as one-point correl measure}
\end{equation}
\end{rem}

Note that the exchangeable random variables $x_{1},..,x_{N}$ are
independent iff $\mu^{(N)}$ is a tensor product measure, $\mu^{(N)}=\mu^{\otimes N},$
where $\mu$ is the law of any $x_{i}.$

\subsubsection{\label{subsec:LDP}Large Deviation Principles (LDP)}

The notion of a \emph{Large Deviation Principle (LDP)}, introduced
by Varadhan, allows one to give a notion of exponential convergence
in probability. The general definition of a Large Deviation Principle
(LDP) for a general sequence of measures \cite{d-z} is modeled on
the classical Laplace method of ``saddle point approximation'' of
integrals:
\begin{defn}
\label{def:large dev}Let $\mathcal{Y}$ be a Polish space, i.e. a
complete separable metric space. A sequence $\Gamma_{k}$ of measures
on $\mathcal{Y}$ satisfies a \emph{large deviation principle} with
\emph{speed} $r_{k}$ and \emph{rate function} $I:\mathcal{\,Y}\rightarrow]-\infty,\infty]$
if

\begin{equation}
\limsup_{k\rightarrow\infty}\frac{1}{r_{k}}\log\Gamma_{k}(\mathcal{F})\leq-\inf_{\mu\in\mathcal{F}}I\label{eq:limsup}
\end{equation}
 for any closed subset $\mathcal{F}$ of $\mathcal{Y}$ and 
\[
\liminf_{k\rightarrow\infty}\frac{1}{r_{k}}\log\Gamma_{k}(\mathcal{G})\geq-\inf_{\mu\in G}I(\mu)
\]
 for any open subset $\mathcal{G}$ of $\mathcal{Y}.$ 
\end{defn}

In the present setting $\Gamma_{N}$ will arise as the sequence of
probability measures on $\mathcal{P}(X)$ defined as laws of the empirical
measures $\delta_{N}$ (formula \ref{eq:empirical measure text}).
Once the LDP has been established we can apply the following basic
\begin{lem}
\label{lem:LDP plus unique implies conv in law}Let $Y_{N}$ be a
sequence of random variables taking values in space $\mathcal{Y}$
which is a compact Polish space. If the laws $\Gamma_{N}\in\mathcal{P}(\mathcal{Y})$
of $Y_{N}$ satisfy a LDP at speed $N$ with a good rate functional
$I$ which admits a\emph{ unique }minimizer $y_{*},$ then $Y_{N}$
converge in law towards $y_{*}.$ More precisely, 
\[
\mbox{Prob}\{d(y_{N},y_{*})\geq\epsilon\}\leq C_{\epsilon}e^{-N/C_{\epsilon}}
\]
\end{lem}

\begin{proof}
We recall the simple proof. First applying the LDP to $\mathcal{F}=\mathcal{G}=\mathcal{Y}$
gives $I(y_{*})=0.$ Since $y_{*}$ is the unique minimizer of $I$
it follows that $\inf I>0$ on the closed subset $\mathcal{F}_{\epsilon}$
of $\mathcal{Y}$ where $d(\cdot,y_{*})\geq\epsilon.$ Applying the
upper bound \ref{eq:limsup} in the LDP to $\mathcal{F}_{\epsilon}$
thus concludes the proof of the desired deviation inequality.
\end{proof}
In other words, the lemma says that risk that $Y_{N}$ deviates from
$y_{*}$ is exponentially small. In order to establish the LDP we
will have great use for the following alternative formulation of a
LDP (see Theorems 4.1.11 and 4.1.18 in \cite{d-z}):
\begin{prop}
\label{prop:d-z}Let $\mathcal{Y}$ be a compact  metric space and
denote by $B_{\epsilon}(y)$ the ball of radius $\epsilon$ centered
at $y\in\mathcal{Y}.$ Then a sequence $\Gamma_{N}$ of probability
measures on $\mathcal{P}$ satisfies a LDP with speed $r_{N}$ and
a rate functional $I$ iff 
\begin{equation}
\lim_{\epsilon\rightarrow0}\liminf_{N\rightarrow\infty}\frac{1}{r_{N}}\log\Gamma_{N}(B_{\epsilon}(y))=-I(y)=\lim_{\epsilon\rightarrow0}\limsup_{N\rightarrow\infty}\frac{1}{r_{N}}\log\Gamma_{N}(B_{\epsilon}(y))\label{eq:ldp in terms of balls in prop}
\end{equation}
\end{prop}

In the present setting of a sequence of random point process with
$N$ particles the previous proposition may be symbolically summarized
as follows: 
\[
\text{Prob }\left(\frac{1}{N}\sum_{i=1}^{N}\delta_{x_{i}}\in B_{\epsilon}(\mu)\right)\sim e^{-r_{N}I(\mu)}
\]
when first $N\rightarrow\infty$ and then $\epsilon\rightarrow0.$

We also recall the following classical result of Sanov, which is the
standard example of a LDP for random point processes (describing the
case when the particles $x_{1},...,x_{N}$ define independent variables
with identical distribution $\mu_{0}$):
\begin{thm}
(Sanov) \label{thm:sanov}Let $X$ be a topological space and $\mu_{0}$
a finite measure on $X.$ Then the laws $\Gamma_{N}$ of the empirical
measures $\delta_{N}$ defined with respect to the product measure
$\mu_{0}^{\otimes N}$ on $X^{N}$ satisfy an LDP with speed $N$
and rate functional the relative entropy $D_{\mu_{0}}.$ 
\end{thm}

\begin{proof}
As explained in \cite{d-z} this is a consequence of the general Gärtner-Ellis
theorem (recalled in Section \ref{subsec:A-general-LDP}). From this
point of view the rate functional $I$ arises as the Legendre-Fenchel
transform of the functional $f(u)$ on $C(X)$ defined by $f(u):=\log\int e^{u}\mu_{0},$
which, by Jensen's inequality is given by $D_{\mu_{0}}.$
\end{proof}
Recall that the \emph{relative entropy} $D_{\mu_{0}}$ (also called
the \emph{Kullback--Leibler divergence }or the\emph{ information
divergence} in probability and information theory) is the functional
on $\mathcal{P}(X)$ defined by 
\begin{equation}
D_{\mu_{0}}(\mu):=\int_{X}\log\frac{\mu}{\mu_{0}}\mu,\label{eq:def of rel entropy}
\end{equation}
 when $\mu$ has a density $\frac{\mu}{\mu_{0}}$ with respect to
$\mu_{0}$ and otherwise $D_{\mu_{0}}(\mu):=\infty.$ If $\mu_{0}$
is a probability measure, then $D_{\mu_{0}}(\mu)\geq0$ and $D_{\mu_{0}}(\mu)=0$
iff $\mu=\mu_{0}$ (by Jensen's inequality). 
\begin{rem}
\label{rem:The-entropy-is}The ``physical entropy'' is usually defined
as 
\[
S(\mu):=-D_{\mu_{0}}(\mu)
\]
In fact, Sanov's theorem can be seen as a mathematical justification
of Boltzmann's original formula expressing the physical entropy $S$
as the logarithm of the number of microscopic states consistent with
a given macroscopic state (using the characterization of a LDP in
Prop \ref{prop:d-z}). 
\end{rem}

\subsection{\label{subsec:Variational-analysis}Variational analysis}

The notion of Gamma-convergence was introduced by de Georgi (see the
book \cite{bra} for background on Gamma-convergence).
\begin{defn}
\label{def:gamma}A sequence of functions $E_{N}$ on a topological
space \emph{$\mathcal{M}$ is said to Gamma$-$converge }to a function
$E$ on $\mathcal{M}$ if 
\begin{equation}
\begin{array}{ccc}
\mu_{N}\rightarrow\mu\,\mbox{in\,}\mathcal{M} & \implies & \liminf_{N\rightarrow\infty}E_{N}(\mu_{N})\geq E(\mu)\\
\forall\mu & \exists\mu_{N}\rightarrow\mu\,\mbox{in\,}\mathcal{M}: & \lim_{N\rightarrow\infty}E_{N}(\mu_{N})=E(\mu)
\end{array}\label{eq:def of gamma conv}
\end{equation}
Given $\mu,$ a sequence $\mu_{N}$ as in the last point above is
called a\emph{ recovery sequence for $\mu.$} The limiting functional
$E$ is automatically lower semi-continuous on $\mathcal{M}.$ 
\end{defn}

\begin{example}
In our complex-geometric setting we will first embed $X^{N}/S_{N}$
into $\mathcal{P}(X)$ using the empirical measure $\delta_{N}$ and
define $E_{N}$ by formula \ref{eq:def of E N intro}, extended by
$\infty$ to all of the space $\mathcal{M}(X)$ of signed measures
on $X.$ Then, as explained in Section \ref{sec:The-case-of pos beta},
$E_{N}$ Gamma-converges towards the pluricomplex energy $E(\mu).$
This example illustrates that, in general, it is not the case that
$\limsup_{N}E_{N}(\mu_{N})\leq E(\mu).$ Indeed, for any sequence
where two points, say $x_{1}$ and $x_{2},$ coincide, $E_{N}(\mu_{N})=\infty!$
For this reason Gamma-convergence is not preserved under multiplication
by negative numbers.
\end{example}

Following \cite{berm10} we will also have use for a weaker notion
of convergence. Given a subset $\mathcal{S}\Subset\mathcal{X}$ we
will say that $f_{j}$ \emph{Gamma-converge to $f$ relative to $\mathcal{S}$}
if the existence of a recovery sequence in $\mathcal{X}$ is only
demanded when $x\in\mathcal{S}.$ The definition is made so that the
following basic property holds:
\begin{lem}
\label{lem:conv of inf}Let $\mathcal{X}$ be a compact topological
space and assume that $f_{j}$ \emph{Gamma-converges }to $f$ relative
to a set $\mathcal{S}$ containing all minima of $f.$ Then 
\[
\lim_{j\rightarrow\infty}\inf_{\mathcal{X}}f_{j}=\inf_{\mathcal{X}}f
\]
Moreover, if $f$ admits a unique minimizer $x,$ then any sequence
$x_{i}$ of minimizers of $f_{j}$ converges towards $x,$ as $j\rightarrow\infty.$
\end{lem}

A general criterion for Gamma-convergence on $\mathcal{P}(X)$ can
be obtained using duality in topological vector spaces, as next explained.
Let $f$ be a function on a topological vector space $V.$ The \emph{Legendre-Fenchel
transform} $f^{*}$ of $f$ is the following convex lower semi-continuous
function $f^{*}$ on the topological dual $V^{*}$ 
\begin{equation}
f^{*}(w):=\sup_{v\in V}\left\langle v,w\right\rangle -f(v)\label{eq:def of Leg-F trans}
\end{equation}
in terms of the canonical pairing between $V$ and $V^{*}.$ In the
present setting we will take $V=C^{0}(X)$ and $V^{*}=\mathcal{M}(X),$
the space of all signed Borel measures on a compact topological space
$X.$ 
\begin{prop}
\label{prop:crit for gamma conv}Let $E_{N}$ be a sequence of functions
on the space $\mathcal{P}(X)$ of probability measures on a compact
space $X$ and assume that
\[
\lim_{N\rightarrow\infty}E_{N}^{*}(w)=f(w)
\]
for any $w\in C(X)$ and that $f$ defines a Gateaux differentiable
function on $C(X).$ Then $E_{N}$ converges to $E:=f^{*}$ in the
sense of Gamma-convergence on the space $\mathcal{P}(X),$ equipped
with the weak topology. 
\end{prop}

See \cite[Prop 4.4]{berm8} for the proof.

\section{\label{sec:The-thermodynamical-formalism}The thermodynamical formalism}

In this section we recall the thermodynamical formalism introduced
in \cite{berm6} and further developed in \cite{bbegz}. The main
character is the free energy functional $F_{\beta}$ on $\mathcal{P}(X),$
which, from the probabilistic point of view, appears as the rate functional
of the LDP for the random point processes. In particular, the lower
semi-continuiuty properties of $F_{\beta}$ will play a crucial role
in Section \ref{sec:Towards-the-case of neg}. 
\begin{rem}
Regardless of the probabilistic motivations the thermodynamical formalism
(in particular, the use of Legendre-Fenchel transforms) also sheds
some new light on the standard functionals in Kähler geometry. This
was, in particular, exploited in the singular setting of log Fano
varieties in \cite{bbegz} and in the variational approach to the
Yau-Tian-Donaldson conjecture in \cite{bbj}. 
\end{rem}

\subsection{The case $\beta>0$}

Let $X$ be a compact complex manifold. To the data $(dV,\omega_{0})$
consisting of a volume form $dV$ and a Kähler form $\omega_{0}$
on $X$ and a parameter $\beta>0$ we attach the following \emph{free
energy functional} on $\mathcal{P}(X)$
\begin{equation}
F_{\beta}(\mu):=\beta E_{\omega_{0}}(\mu)+D_{dV}(\mu)\in]-\infty,\infty]\label{eq:def of F beta by form}
\end{equation}
which is lsc and convex on $\mathcal{P}(X)$ (since both terms are). 
\begin{rem}
In thermodynamics the free energy is usually defined as $E+\beta^{-1}D_{dV}$
(since it is the energy that is free to perform work after the ``useless''
thermal energy has been subtracted; compare Remark \ref{rem:The-entropy-is}).
This definition was also used in \cite{berm6}. But here it will be
convenient to use the rescaled version of the free energy above, since
it facilitates the transition from positive to negative $\beta.$ 
\end{rem}

\begin{lem}
\label{lem:diff of F}Assume that $\mu_{\beta}$ is a volume form.
Then the differential of $F_{\beta}$ at $\mu$ is given by 
\[
dF_{\beta}(\mu)=-\beta\varphi_{\mu}+\log\frac{\mu}{dV},
\]
i.e. if $\mu_{t}$ is an affine curve of volume forms in $\mathcal{P}(X),$
then 
\[
\frac{dF_{\beta}(\mu_{t})}{dt}=\left\langle -\beta\varphi_{\mu_{t}}+\log\frac{\mu_{_{t}}}{dV},\frac{d\mu_{t}}{dt}\right\rangle 
\]
\end{lem}

\begin{proof}
Since $dD(\mu)=\log\frac{\mu}{dV}$ this follows directly from formula
\ref{eq:differ of E}.
\end{proof}
As a consequence, if $\mu_{\beta}$ is a volume form which is a critical
point of $F_{\beta}$ on $\mathcal{P}(X),$ then $\mu_{\beta}$ satisfies
the mean field type equation
\[
-\beta\varphi_{\mu_{\beta}}+\log\frac{\mu_{\beta}}{dV}+\log Z=0
\]
 for some positive constant $Z,$ i.e. 
\begin{equation}
\mu_{\beta}=\frac{e^{\beta\varphi_{\mu_{\beta}}}dV}{Z}\label{eq:crit point eq for F}
\end{equation}
Equivalently, this means that the potential $\varphi_{\beta}$ of
$\mu_{\beta}$ satisfies, after perhaps shifting $\varphi_{\mu}$
by a constant to ensure $Z=1,$ the complex Monge-Ampère equation
\ref{eq:ma eq with beta in text}: 
\begin{equation}
MA(\varphi)=e^{\beta\varphi_{\mu_{\beta}}}dV\label{eq:MA eq for beta in section thermo}
\end{equation}
 To a volume form $\mu_{\beta}$ minimizing $F_{\beta}$ on $\mathcal{P}(X)$
we attach the Kähler potential 
\[
\varphi_{\beta}:=\frac{1}{\beta}\log\frac{\mu_{\beta}}{dV}
\]
and the Kähler form on $X$
\[
\omega_{\beta}:=\omega_{0}+dd^{c}\varphi_{\beta},
\]
 which satisfies the twisted Kähler-Einstein equation \ref{eq:tw ke eq in text}. 

The following result was shown in \cite{berm6} (see also \cite{berm8 comma 5}
for singular generalizations):
\begin{thm}
\label{thm:free energy beta pos}Assume given $(dV,\omega_{0})$ as
above and parameter $\beta>0.$ Then the free energy functional $F_{\beta}(\mu)$
admits a unique minimizer $\mu_{\beta}$ on $\mathcal{P}(X).$ Moreover,
$\mu_{\beta}$ is a volume form and the corresponding Kähler potential
$\varphi_{\beta}$ and Kähler form $\omega_{\beta}$ satisfy the complex
Monge-Ampère equation \ref{eq:ma eq with beta in text} and twisted
twisted Kähler-Einstein equation \ref{eq:tw ke eq in text}, respectively. 
\end{thm}

\begin{proof}
By the Aubin-Yau Theorem \ref{thm:(Aubin-Yau)--Assume} there exists
a volume form $\mu_{\beta}$ solving the critical point equation \ref{eq:crit point eq for F}.
Indeed, we can take $\mu_{\beta}:=MA(\varphi_{\beta}),$ where $\varphi_{\beta}$
solves the complex MA-equation in Theorem \ref{thm:(Aubin-Yau)--Assume}.
But $E$ is convex on $\mathcal{P}(X)$ (by its very definition as
a sup of affine functions) and $D$ is strictly convex (by Jensen's
inequality). Hence, it is enough to show that $\mu_{\beta}$ is a
sub-gradient for $F_{\beta}.$ But this follows from the sub-gradient
relation\ref{eq:sub-grad} . 
\end{proof}
\begin{rem}
Instead of relying on the Aubin-Yau theorem the minimizer $\mu_{\beta}$
can be obtained by directly maximizing the functional $\mathcal{G}_{\beta},$
appearing in the following section, following the variational approach
introduced in \cite{bbgz}. This alternative approach is important
in the more general case of a big cohomology class, as well as in
singular settings (see \cite{berm8 comma 5}).
\end{rem}

\subsection{The case $\beta<0$}

In the case when $\beta<0$ we define $F_{\beta}(\mu)$ by the same
expression \ref{eq:def of F beta by form}, when $E_{\omega_{0}}(\mu)<\infty$
and otherwise we set $F_{\beta}(\mu)=\infty.$ The definition is made
so that we still have $F_{\mu}(\mu)\in]-\infty,\infty]$ with $F_{\mu}(\mu)<\infty$
iff both $E(\mu)<\infty$ and $D(\mu)<\infty$ . Set 
\[
\gamma=-\beta>0.
\]
In order to the study the functional $F_{-\gamma}$ the following
auxiliary ``dual'' functional on $C^{0}(X)$ turns out to be very
useful:

\[
\mathcal{G}_{-\gamma}(u):=\mathcal{E}(P(u))-\frac{1}{\gamma}\log\int_{X}e^{-\gamma u}dV:\,\,C^{0}(X)\rightarrow\R
\]
Moreover, for $\varphi\in\mathcal{E}^{1}(X)$ we set 
\[
\mathcal{G}_{-\gamma}(\varphi):=\mathcal{E}(\varphi)-\frac{1}{\gamma}\log\int_{X}e^{-\gamma\varphi}dV:\,\,\mathcal{E}^{1}(X)\rightarrow\R,
\]
 which is consistent with the previous notation since both functional
coincide on the intersection of their domains. Note that the critical
point equation for the functional $\mathcal{G}_{-\gamma}(\varphi)$
is precisely the Monge-Ampère equation \ref{eq:MA eq for beta in section thermo}.
Moreover, as observed in \cite{berm6} the functional $\mathcal{G}_{-\gamma}$
on $C^{0}(X)$ may be expressed in terms of termwise Legendre-Fenchel
transforms of $F_{\beta},$ as exploited in the proof of the following
result from \cite{berm6}:
\begin{lem}
\label{lem: inf f g etc}The following holds: 
\[
\inf_{\mathcal{P}(X)}\gamma F_{-\gamma}=-\sup_{C^{0}(X)}\mathcal{G}_{-\gamma}=-\sup_{\mathcal{E}^{1}(X)}\mathcal{G}_{-\gamma}
\]
Moreover, for any $\varphi\in\mathcal{E}^{1}(X)$
\[
\gamma F_{-\gamma}(MA(\varphi))\geq-\mathcal{G}_{-\gamma}(\varphi)
\]
\end{lem}

\begin{proof}
Let $f$ and $g$ be two lsc convex function on the dual $\mathcal{X}^{*}$
of a locally convex topological vector space $\mathcal{X}.$ Then
\begin{equation}
\inf_{\mathcal{X}}\left(g^{*}-f^{*}\right)=-\sup_{\mathcal{X}^{*}}\left(g-f\right)\label{eq:in f g}
\end{equation}
Indeed, this follows directly from the fact that the Legendre-Fenchel
transform is decreasing and involutive on $\mathcal{X}$ (see Section
4.5.2 in \cite{d-z}). Moreover, if $f$ is Gateaux differentiable,
then, for any $u\in\mathcal{X^{*}},$
\begin{equation}
(g^{*}-f^{*})(df[u])\geq-\left(g-f\right)(u),\label{eq:ineq f g}
\end{equation}
just using that $g^{*}(df[u])\geq\left\langle df[u],u\right\rangle -g(u)$
and $f(u)=\left\langle df[u],u\right\rangle -f(u).$ In particular,
taking $\mathcal{X}$ to be the space $\mathcal{M}(X)$ of all signed
measures on $X$ and setting $g(u):=\frac{1}{\gamma}\log\int_{X}e^{\gamma u}dV$
on $C^{0}(X)=\mathcal{M}(X)^{*}$ gives $g^{*}=\gamma D$ (compare
the proof of Theorem \ref{thm:sanov}). Moreover, by Prop \ref{prop:pluri energy as legendre},
$E=f^{*}$ for $f(u):=-\mathcal{E}(P(-u)$ and hence applying formula
\ref{eq:in f g} yields
\[
\inf_{\mathcal{P}(X)}\left(\gamma D-E\right)=-\sup_{C^{0}(X)}\left(\frac{1}{\gamma}\log\int_{X}e^{\gamma u}\mu+\mathcal{E}(P(-u))\right),
\]
 which proves the first equality in the lemma (since the sup in the
rhs above is invariant under $u\mapsto-u$). The second equality then
follows from combining $P(u)\leq u$ with monotonicity and a simple
approximation argument. Finally, the inequality \ref{eq:ineq f g},
combined with the differentiability Theorem \ref{thm:diff}r, yields
the inequality in the lemma when $\varphi\in\mathcal{H}(X)$ and the
general case then follows by approximation.

The previous lemma is used in the proof of the following regularity
result from \cite{bbegz}:
\end{proof}
\begin{thm}
\label{thm:(Regularity).-Any-miniminizer}(Regularity). Any minimizer
$\mu_{\beta}$ of $F_{\beta}$ is a volume form and hence the corresponding
Kähler form $\omega_{\beta}$ satisfies the twisted Kähler-Einstein
equation \ref{eq:tw ke eq in text}.
\end{thm}

\begin{proof}
If $\mu_{\beta}$ minimizes $F_{\beta}$ then, in particular, $E(\mu_{\beta})<\infty.$
Let $\varphi_{\beta}\in$$\mathcal{E}^{1}(X)$ be a potential for
$\mu_{\beta},$ i.e. $MA(\varphi_{\beta})=\mu.$ It follows from the
previous lemma that $\varphi_{\beta}$ maximizes the functional $G_{\beta}$
on $\mathcal{E}^{1}(X).$ But then it follows, as shown in \cite{bbgz},
that $\phi_{\beta}$ satisfies the MA-equation \ref{eq:MA eq for beta in section thermo}
(using the differentiability of the functional $\mathcal{E}\circ P$
on $\mathcal{E}^{1}(X)+C^{0}(X),$ which follows from Theorem \ref{thm:diff}).
Finally, as shown in the appendix of \cite{bbegz} (using Aubin-Yau
type Laplacian estimates) any solution in $\mathcal{E}^{1}(X)$ is
smooth, as desired. 
\end{proof}
It can be shown that $F_{\beta}$ is not bounded from below for $\beta$
sufficiently negative. In particular, it does not have a minimizer
then. But we have the following 
\begin{thm}
\label{thm:F lsc for beta bigger}Assume that the free energy functional
$F_{\beta_{0}}$ is bounded from below for some $\beta_{0}<0.$ Then
for any $\beta>\beta_{0}$ the functional $F_{\beta}$ is lsc on $\mathcal{P}(X).$
In particular, it then admits a minimizer $\mu_{\beta}\in\mathcal{P}(X).$
\end{thm}

\begin{proof}
The lower semi-continuity follows from the results in \cite{bbegz},
as we next recall. Take a sequence $\mu_{j}\rightarrow\mu_{\infty}$
in $\mathcal{P}(X).$ We may as well assume that $F_{\beta}(\mu_{j})\leq C<\infty$
(otherwise there is nothing to prove). Since $F_{\beta_{0}}\geq C_{0}$
we get $D(\mu_{j})\leq C_{1}.$ But then it follows the energy-entropy
compactness theorem in \cite[Thm 2.17]{bbegz} that $\mu_{j}\rightarrow\mu_{\infty}$
in $\mathcal{P}(X)$ and $E(\mu_{j})\rightarrow E(\mu_{\infty})<\infty.$
Since $D$ is lsc we deduce that $F_{\beta}(\mu_{\infty})\leq\liminf_{j}F_{\beta}(\mu_{j}),$
as desired. 
\end{proof}
In general, a minimizer of $F_{\beta}$ need not be unique when $\beta<0$
and moreover there may be critical points which are not minimizers.
However, the situation simplifies in the ``Fano setting'' (Section
\ref{subsec:The-Fano-setting}).
\begin{thm}
\label{thm:free energy Fano}Consider the ``Fano setting'' and fix
$\beta\in[-1,0[$. In the case $\beta=-1$ we assume that the group
$\text{Aut \ensuremath{(X)_{0}}}$ is trivial. Then the functional
$F_{\beta}$ admits at most one minimizer. Moreover, the following
is equivalent
\begin{enumerate}
\item There exists a minimizer $F_{\beta}$ on $\mathcal{P}(X)$
\item There exists $\epsilon>0$ such that the functional $F_{\beta-\epsilon}$
is bounded below on $\mathcal{P}(X),$ i.e. the following \emph{connectivity
inequality} holds for some constant $C_{\epsilon}$ 
\begin{equation}
F_{\beta}(\mu)\geq\epsilon E_{\omega_{0}}-C_{\epsilon}\label{eq:coerc ineq for F}
\end{equation}
\item $F_{\beta}$ is lsc on $\mathcal{P}(X)$ 
\end{enumerate}
\end{thm}

\begin{proof}
Combining the regularity theorem with the uniqueness Theorem \ref{thm:B-M}
for solutions to the twisted Kähler-Einstein-equation for $\omega_{\beta}$
shows that $F_{\beta}$ admits at most one minimizer. By the previous
theorem all that remains is the implication ``$1\implies2".$ We
thus assume that there exists a minimizer $\mu_{\beta}.$ By the regularity
Theorem \ref{thm:(Regularity).-Any-miniminizer} this means that there
exists a Kähler metric $\omega_{\beta}$ solving the twisted Kähler-Einstein-equation
\ref{eq:tw ke eq in text}. Since the restriction of $F_{\beta}$
to the space of volume forms may be identified with the twisted Mabuchi
functional (see Section \ref{subsec:Relation-to-the standard K})
the connectivity inequality \ref{eq:coerc ineq for F} restricted
to the space of volume forms then follows from the corresponding coercivity
inequality for the twisted Mabuchi functional on $\mathcal{H}(X,\omega_{0}),$
established in \cite{ti} Tian (using Aubin's method of continuity).
Finally, if $\mu$ satisfies $F_{\beta}(\mu)<\infty$ we take a sequence
of volume forms $\mu_{j}$ converging weakly towards $\mu$ such that
$D(\mu_{j})\rightarrow D(\mu).$ By the energy-entropy compactness
theorem, as used in the proof of Theorem \ref{thm:F lsc for beta bigger},
$E(\mu_{j})\rightarrow E(\mu)$ and hence the coercivity inequality
holds on all of $\mathcal{P}(X)$.
\end{proof}
\begin{rem}
In the case when $\beta=-11$ and the group $\text{Aut \ensuremath{(X)_{0}}}$
it follows from \cite{d-r} that the equivalence between the first
two items still holds if the lower bound $E_{\omega_{0}}(\mu)$ is
replaced by $\inf_{f}E_{\omega_{0}}(f_{*}\mu),$ where $f$ ranges
over all element in $\text{Aut \ensuremath{(X)_{0}}. }$
\end{rem}

We next observe that the lower-semi continuity of $F_{-1},$ in fact,
forces the group $\text{Aut \ensuremath{(X)_{0}}}$ to be trivial,
leading to the following equivalence:
\begin{thm}
\label{thm:free energy beta minus one Fano}Let $X$ be a Fano manifold.
Then the following is equivalent:
\begin{itemize}
\item $X$ admits a unique Kähler-Einstein metric
\item The free energy functional $F_{-1}$ is lsc on $\mathcal{P}(X)$
\end{itemize}
\end{thm}

\begin{proof}
By the previous theorem we just have to show that the second point
implies the first point.We thus assume that $F_{-1}$ is lsc on $\mathcal{P}(X).$
By the previous theorem this implies that $X$ admits a Kähler-Einstein
metric. Hence, by the uniqueness Theorem \ref{thm:B-M} we just have
to show that the group $\text{Aut \ensuremath{(X)_{0}}}$ is non-trivial.
Assume, to get a contradiction, that this is not the case. Then $X$
admits a non-trivial $\C^{*}-$action (since $\text{Aut \ensuremath{(X)_{0}}}$
is reductive when $X$ admits a Kähler-Einstein-metric). Denote by
$\rho_{\tau}$ the corresponding family in $\text{Aut \ensuremath{(X)_{0}}},$
parameterized by $\tau\in\C^{*}.$ Fix any volume form $dV_{1}$ on
$X$ and set $dV_{\tau}:=(\rho_{\tau})_{*}dV.$ Note that $F(dV_{\tau})$
is independent of $\tau.$ Indeed, as recalled below $F(\mu)$ can
be identified with the Mabuchi functional $\mathcal{M}$ and it is
well-known that $\mathcal{M}$ is invariant under any $\C^{*}-$action,
if there exists a Kähler-Einstein-metric. It follows from standard
results that when $\tau\rightarrow0$ the volume forms $dV_{\tau}$
converge weakly to a measure $\mu_{0}\in\mathcal{P}(X)$ supported
on the fixed-point locus $Z$ in $X$ of the $\C^{*}-$action. In
particular, $\mu_{0}$ is supported on a proper analytic subvariety
of $X.$ But, for any measure $\mu$ charging a pluripolar subset
$E_{\omega_{0}}(\mu)=\infty$ \cite{bbgz}. In particular, $F(\mu_{0})=\infty$
and hence, $F$ cannot be lsc along the family $dV_{\tau}.$ For completeness,
we note that the converge of $dV_{t}$ towards $\mu_{0}$ used above,
can be shown as follows. First note that $\rho_{0}(x):=\lim_{\tau\rightarrow0}\rho_{\tau}(x)$
yields a well-defined continuous map $\rho_{0}$ from $X-Z$ into
$Z.$ Indeed, using that the $\C^{*}$ action lifts to $-K_{X}$ and
that the Kodaira embedding \ref{eq:Kod emb} is $\C^{*}-$equivariant
we can identify $X$ with a submanifold of $\P^{m}$ and the $\C^{*}-$action
with the restriction to $X$ of a linear $\C^{*}-$action on $\P^{m}.$
The existence of $\rho_{\infty}(x)$ then follows from the projective
case, where the map in question is simply a rational projection from
$\P^{m}$ onto a projective linear subspace. Hence, the limit $\mu_{0}$
is equal to the push-forward of $dV$ under the map $\rho_{\infty}$
(which is well-defined since $dV$ does not charge the locus $Z$
where $f_{\infty}$ is undetermined).
\end{proof}

\subsubsection{\label{subsec:The-log-Fano-setting thermo}The log Fano setting}

Now assume given a log Fano manifold $(X,\Delta),$ as defined in
Section \ref{subsec:The-log-Fano}. Assume that $(X,\Delta)$ is (sub-)klt.
We then consider the data $(\omega_{0},\mu_{0})$ consisting of a
Kähler metric $\omega_{0}\in c_{1}(K_{X}+\Delta)$ and the induced
measure $\mu_{0}$ on $X$ with singularities along $\Delta,$ defined
as in Section \ref{subsec:Log-Fano-manifolds}. In this setting we
define the corresponding functional $F_{\beta}$ as in the Fano setting,
but replacing the volume form $dV$ in formula \ref{eq:def of F beta by form}
by the measure $\mu_{0}.$ The following results are shown in \cite{bbegz}.
\begin{thm}
\label{thm:log Fano free energy}Consider a log Fano manifold $(X,\Delta).$
If the functional $F_{\beta}$ is bounded from below for some $\beta<-1,$
then $F_{-1}$ admits a minimizer $\mu_{-1}.$ Moreover, any minimizer
$\mu_{-1}$ is given by a volume form on $X-\Delta$ and 
\[
\omega_{-1}:=\omega_{0}-dd^{c}\log\frac{\mu_{-1}}{\mu_{0}}
\]
defines a Kähler metric on $X-\Delta$ which extends to a current
in $-c_{1}(K_{(X,\Delta)})$ solving the Kähler-Einstein-equation
for $(X,\Delta).$ In the case when $\Delta$ is klt the minimizer
is uniquely determined.
\end{thm}

The uniqueness statement in the previous theorem follows from the
generalization of the Bando-Mabuchi Theorem \ref{thm:B-M} to general
singular log Fano varieties in \cite{bbegz}, deduced from a variant
of the convexity properties of the Ding functional along bounded geodesics
in $PSH(X,\omega_{0})\cap L^{\infty}(X)$ established in \cite{bern2}
(see also \cite[the appendix of III]{c-d-s}).

\subsection{\label{subsec:Relation-to-the standard K}Relation to the standard
functionals in Kähler geometry }

First observe that it follows directly from formula \ref{eq:expres for E mu in terms of potential},
that if $\varphi)\in\mathcal{H}(X,\omega_{0}),$ then we can express
\[
E_{\omega_{0}}(MA(\varphi))=(I_{\omega_{0}}-J_{\omega_{0}})(\varphi),
\]
 where $I_{\omega_{0}}$ and $J_{\omega_{0}}$ are the standard energy
type functionals in Kähler geometry defined by
\[
I_{\omega_{0}}(\varphi):=-\frac{1}{(n+1)}\int\varphi MA(\varphi)
\]
 and 
\[
I_{\omega_{0}}(\varphi):=\frac{1}{V}\int\varphi\omega_{0}^{n}-\mathcal{E}_{\omega_{0}}(\varphi)
\]
As is well-known all functionals $I_{\omega_{0}},$ $J_{\omega_{0}}$
and $I_{\omega_{0}}-J_{\omega_{0}}$ are non-negative (vanishing only
for $\varphi=0)$ and mutually compatible, with constants only depending
on $n$ (see, for example, \cite{bbgz}). 

In the ``Fano setting'' with $\beta=-1$ the functional 
\[
\mathcal{D}(\varphi):=-\mathcal{G}_{-1}(\varphi):=-\mathcal{E}_{\omega_{0}}(\varphi)-\log\int_{X}e^{-\varphi}dV
\]
 is the \emph{Ding functional} introduced in \cite{ding}. In the
case of a general $(dV,\omega_{0})$ the corresponding functional
$-\mathcal{G}_{-1}(\varphi)$ can be viewed as a generalization of
the Ding functional to the twisted setting.

Next, consider the functional $\mathcal{M}_{\beta}$ on $\mathcal{H}(X,\omega_{0})$
defined by
\[
F_{\beta}\left(MA(\varphi)\right)=:\mathcal{M}_{\beta}(\varphi)
\]

\begin{lem}
Given the data $(dV,\omega_{0})$ the differential of $\mathcal{M}_{\beta}$
at $\varphi\in\mathcal{H}(X,\omega_{0})$ is given by
\[
d\mathcal{M}_{\beta}(\varphi)=-\frac{1}{V}n\left(\beta\omega_{\varphi}+\text{Ric }\omega_{\varphi}-\theta\right)\wedge\omega_{\varphi}^{n-1}=:-\frac{n}{V}\left(\beta+n(R_{\omega_{\varphi}}-\text{tr}_{\omega_{\varphi}}\theta)\right)\omega_{\varphi}^{n}
\]
 where $R_{u}$ is the (normalized) scalar curvature of the Kähler
metric $\omega_{u}$ and $\text{tr}_{\omega_{u}}\theta$ is the trace
of the twisting form $\theta$ with respect to $\omega_{\varphi}.$
\end{lem}

\begin{proof}
Combining Lemma \ref{lem:diff of F} with the chain rule and using
that $d(MA(u_{t}))/dt=V^{-1}n\omega_{u_{t}}^{n-1}\wedge dd^{c}(du_{t}/dt)$
gives 
\[
d\mathcal{M}_{\beta}(\varphi_{t})/dt=V^{-1}n\left\langle -\beta\varphi_{t}+\log\frac{MA(\varphi_{t})}{dV},\omega_{\varphi_{t}}^{n-1}\wedge dd^{c}\varphi_{t}\right\rangle .
\]
Integrating by parts to move $dd^{c}$to the other side then concludes
the proof, using that $R_{\omega}:=\text{\ensuremath{\text{tr}_{\omega}}Ric }\omega:=n^{-1}\text{Ric}\omega\wedge\omega^{n-1}/\omega^{n}.$ 
\end{proof}
For $\beta=\pm1$ and $\theta=0$ we thus have that $\omega_{0}\in c_{1}(K_{X})$
\begin{equation}
d\mathcal{M}_{\pm}(u)=-\frac{n}{V}\left(\pm\omega_{u}+\text{Ric }\omega_{u}-\theta\right)\wedge\omega_{u}^{n-1}=:-\frac{n}{V}\left((\pm1+nR_{\omega_{u}})\right)\omega_{u}^{n},\label{eq:d M plus min}
\end{equation}
 which is the defining property of the \emph{Mabuchi functional }on
the space of Kähler potentials for $\pm c_{1}(K_{X}),$ introduced
in \cite{ma} (which is thus only defined up to an additive constant).
More generally, when $\beta=\pm1$ and $\theta$ is a general closed
one-form satisfying the cohomological equation \ref{eq:cohom eq}
the formula \ref{eq:d M plus min} is the defining property of the
corresponding twisted Mabuchi functional on the space of Kähler potentials
for $\pm\left(c_{1}(K_{X})+[\theta]\right).$ 

\subsubsection{Digression on constant scalar curvature }

We make a brief digression to recall that the Mabuchi functional $\mathcal{M}_{\omega_{0}}$
is, in fact, defined for\emph{ any }Kähler class $[\omega_{0}]$ by
the property 
\[
d\mathcal{M}_{\omega_{0}}(u)=-\frac{1}{V}n\left(C_{0}\omega_{u}+\text{Ric }\omega_{u}\right)\wedge\omega_{u}^{n-1},
\]
 where $C_{0}$ is the cohomological constant ensuring that the right
hand side above integrates to zero over $X.$ This means that the
critical points of $d\mathcal{M}_{\omega_{0}}$ in in $\mathcal{H}(X,\omega_{0})$
are the Kähler potentials defining Kähler metrics with constant scalar
curvature. Similarly, for any closed $(1,1)-$form $\theta$ there
is a twisted Mabuchi functional on $\mathcal{H}(X,\omega_{0})$ associated
to the twisting form $\theta$ (obtained by replacing $\text{Ric }\omega_{u}$
with $\text{Ric }\omega_{u}-\theta$ and adjusting the cohomological
constant $C_{0}$ accordingly). In this general setup there is variant
of Theorem \ref{thm:free energy Fano} saying that there exists a
unique Kähler metric in $[\omega_{0}]\in H^{2}(X,\R)$ iff the Mabuchi
functional $\mathcal{M}$ is coercive on $\mathcal{H}(X,\omega_{0}).$
This is the content of Tian's properness conjecture, which was recently
settled in \cite{c-c} using a generalization of Aubin's continuity
method. The ``only if'' direction was previously shown in \cite{b-d-l},
building on the general existence/properness principle in metric spaces
established in \cite{d-r}. Its application to Kähler geometry is
based on the metric space realization of $\mathcal{E}^{1}(X)$ introduced
in \cite{da} together with the energy/entropy compactness theorem
in \cite{bbegz} and the geodesic convexity of $\mathcal{M}$ in \cite{b-b}.

\section{\label{sec:The-case-of pos beta}The large $N-$limit in the case
of positive $\beta$}

In this section we will explain the key ideas in the proof in \cite{berm8}
of Theorem \ref{thm:LDP intro} (which implies Theorem \ref{thm:ke intro}
and also shows that Conjecture \ref{conj:Fano beta greater than minus 1}
is valid when $\beta>0).$ 

Let $(X,L)$ be a polarized manifold. Since $L$ is positive the sequence
\[
N_{k}:=\dim H^{0}(X,kL)
\]
 tends to infinity, as $k\rightarrow\infty.$ To the geometric data
$(dV,\left\Vert \cdot\right\Vert )$ consisting of a volume form $dV$
on and a metric $\left\Vert \cdot\right\Vert $ on $L$ we attach,
for any parameter $\beta\in]0,\infty[$ the following sequence of
probability measures $\mu_{\beta}^{(N_{k})}$ on $X^{N_{k}}:$ 
\begin{equation}
\mu_{\beta}^{(N_{k})}:=\frac{\left\Vert \det S^{(k)}(x_{1},...,x_{N_{k}})\right\Vert ^{2\beta/k}}{Z_{N_{k},\beta}}dV^{\otimes N_{k}},\label{eq:temper deformed text}
\end{equation}
 where $\det S^{(k)}$ is the holomorphic section of $(kL)^{\boxtimes N_{k}}\rightarrow X^{N_{k}}$
defined as a Slater determinant for $H^{0}(X,kL):$
\[
(\det S^{(k)})(x_{1},x_{2},...,x_{N}):=\det(s_{i}^{(k)}(x_{j})),
\]
 in terms of a given basis $s_{i}^{(k)}$ in $H^{0}(X,kL).$ A change
of basis only has the effect of multiplying the section $\det S^{(k)}$
by a complex constant $c$ (the determinant of the change of bases
matrix) and hence the probability measure $\mu^{(N_{k})}$ is independent
of the choice of basis. We will fix a basis which is orthonormal with
respect to the scalar product on $H^{0}(X,kL)$ induced by the data
$(dV,\left\Vert \cdot\right\Vert )$. As in Section \ref{subsec:Complex-geometry}
we denote by $\omega_{0}$ the curvature form of the fixed metric
$\left\Vert \cdot\right\Vert $ on $L.$ 
\begin{example}
(``canonical case''). When $L=K_{X}$ and $\left\Vert \cdot\right\Vert $
is taken to be the metric on $K_{X}$ induced by the fixed volume
form $dV,$ the contributions from $\left\Vert \cdot\right\Vert $
and the volume form $dV$ cancel and then $\mu_{\beta}^{(N_{k})}$
coincides with the canonical probability measure defined in Section
\ref{eq:canon prob measure intro}. 

The following result from \cite[Thm 5.7]{berm8}, was stated as Theorem
\ref{thm:LDP intro} in Section \ref{sec:A-bird's-eye-view}.
\end{example}

\begin{thm}
\label{thm:LDP beta pos text}Let $(X,L)$ be a polarized manifold
and fix the geometric data $(dV,\left\Vert \cdot\right\Vert ).$ Then,
for any $\beta>0,$ the the laws of the corresponding random measures
$\delta_{N}$ on $(X^{N},\mu_{\beta}^{(N)})$ satisfy a Large Deviation
Principle (LDP) with speed $N$ and rate functional $F_{\beta}(\mu)-C_{\beta},$
where 
\[
F_{\beta}(\mu)=\beta E_{\omega_{0}}(\mu)+D_{dV}(\mu),\,\,\,C_{\beta}=\inf_{\mathcal{P}(X)}F_{\beta}.
\]
In particular, 
\[
-\lim_{N\rightarrow\infty}N^{-1}\log Z_{N,\beta}=\inf_{\mathcal{P}(X)}F_{\beta}
\]
 
\end{thm}

Since, by Theorem \ref{thm:free energy beta pos}, there exists a
unique minimizer $\mu_{\beta}$ of $F_{\beta}$ it follows (see Lemma
\ref{lem:LDP plus unique implies conv in law}) that $\delta_{N}$
converges in law towards the unique minimizer $\mu_{\beta}.$ We recall
that $\mu_{\beta}$ is a volume form and the Kähler metric 
\[
\omega_{\beta}:=\omega_{0}+\frac{1}{\beta}dd^{c}\log\frac{\mu_{\beta}}{dV}
\]
is the unique solution to the twisted Kähler-Einstein equation 
\[
\mbox{\ensuremath{\mbox{Ric}}\ensuremath{\omega}}=-\beta\omega+\theta,\,\,\,\theta:=\beta\omega_{0}+\ensuremath{\mbox{Ric}}\ensuremath{dV}
\]
 In particular, specializing the previous theorem to the ``canonical
case'' in the previous example yields Theorem \ref{thm:ke intro}.

\subsection{\label{subsec:The-proof-of Thm LDP text}The proof of Theorem \ref{thm:LDP beta pos text}}

As mentioned in Section \ref{subsec:The-proof-of LDP intro} the starting
point of the proof of Theorem \ref{thm:LDP intro} is to rewrite $\mu^{(N_{k})}$
as a Gibbs measure, at inverse temperature $\beta,$
\[
\mu_{\beta}^{(N_{k})}=\frac{e^{-\beta NE^{(N)}}}{Z_{N,\beta}}dV^{\otimes N},\,\,\,E^{(N)}:=-\frac{1}{kN}\log\left\Vert \det S^{(k)}(x_{1},...,x_{N_{k}})\right\Vert ^{2}
\]
 where, $E^{(N)}$ is called the \emph{energy per particle} and the
normalization constant $Z_{N,\beta}$ is called the \emph{partition
function}. To explain the idea of the proof first assume that that
the following ``Mean Field Approximation'' holds in an appropriate
sense
\begin{equation}
E^{(N)}(x_{1},...x_{N})\approx E(\frac{1}{N}\sum_{i=1}^{N}\delta_{x_{i}}),\,\,\,\,N>>1\label{eq:mean field apprx heurstic form}
\end{equation}
for some functional $E$ on $\mathcal{P}(X).$ We are going to use
the characterization \ref{prop:d-z} of a LDP. By definition, given
$\mu\in\mathcal{P}(X)$ and $\epsilon>0$
\[
\text{Prob }\left(\frac{1}{N}\sum_{i=1}^{N}\delta_{x_{i}}\in B_{\epsilon}(\mu)\right):=Z_{N,\beta}^{-1}\int_{\delta_{N}^{-1}\left(B_{\epsilon}(\mu)\right)}e^{-\beta NE^{(N)}}dV^{\otimes N}
\]
Hence, formally, as $N\rightarrow\infty$ and $\epsilon\rightarrow0,$
we can take out the factor $e^{-\beta NE^{(N)}}$ to get
\begin{equation}
\int_{\delta_{N}^{-1}\left(B_{\epsilon}(\mu)\right)}e^{-\beta NE^{(N)}}dV^{\otimes N}\sim e^{-\beta NE(\mu)}\int_{\delta_{N}^{-1}\left(B_{\epsilon}(\mu)\right)}dV^{\otimes N}\label{eq:take out}
\end{equation}
Applying the Sanov's LDP result \ref{thm:sanov} to the integral thus
suggests that the non-normalized measures 
\[
(\delta_{N})_{*}\left(e^{-\beta H^{(N)}}dV^{\otimes N}\right)
\]
on $\mathcal{P}(X)$ satisfy a LDP with speed $N$ and rate functional
\[
F_{\beta}(\mu):=E(\mu)+\beta^{-1}D_{dV}(\mu).
\]
Once this LDP has been established the asymptotics for $Z_{N}$ follow
from the very definition of a LDP.

\subsubsection{The two technical ingredients}

In order to make this argument rigorous two issues need to be confronted.
First, the nature of the convergence in the ``Mean Field Approximation''
\ref{eq:mean field apprx heurstic form} has to be specified. Secondly,
appropriate conditions on $E^{(N)}$ need to be introduced, ensuring
that the ``taking out'' argument \ref{eq:take out} is justified.
As for the first issue it is shown in \cite{berm8}, that, the approximation
\ref{eq:mean field apprx heurstic form} holds in the sense of\emph{
}Gamma-convergence, with $E$ given by the pluricomplex energy $E_{\omega_{0}}$
(defined by formula \ref{eq:def of e as sup}). More precisely, using
the embedding 
\begin{equation}
\delta_{N}:\,X^{N}/S_{N}\rightarrow\mathcal{P}(X)\label{eq:embedding delta in proof LDP}
\end{equation}
 we can identify $E^{(N)}$ with a function on $\mathcal{P}(X),$
defined to be equal to $\infty,$ on the complement of the image of
$\delta_{N}$ in $\mathcal{P}(X).$ Under this identification it is
shown in \cite{berm8} that the following Gamma-convergence on $\mathcal{P}(X)$
holds:
\begin{equation}
E^{(N)}\rightarrow E_{\omega_{0}},\,\,\,N\rightarrow\infty\label{eq:Gamma conv towards E text}
\end{equation}
 In fact, using the dual criterion in Prop \ref{prop:crit for gamma conv},
this follows directly from the Legendre-Fenchel formula for $E_{\omega_{0}}$
in Prop \ref{prop:pluri energy as legendre} and the differentiability
Theorem\ref{thm:diff}, combined with the following convergence result
for the weighted transfinite diameters of $X$ in \cite{b-b}: given
$u\in C^{0}(X)$
\begin{equation}
\frac{1}{kN_{k}}\log\left\Vert \left(\det S^{(k)}\right)e^{-u/2}\right\Vert _{L^{\infty}(X^{N_{k})})}^{2}\rightarrow\mathcal{E}_{\omega_{0}}(Pu),\,\,\,k\rightarrow\infty,\label{eq:asympt of log det L infty}
\end{equation}
 where we are using the same notation $u$ for the induced function
$\sum_{i}u(x_{i})$ on $X^{N_{k}}.$ As for the ``taking out'' issue
in formula \ref{eq:take out} it is handled using the following key
asymptotic sub-mean inequality in high dimensions established in \cite{berm8}.
There exist positive constants $C$ and $A_{\epsilon}$ such that
for any $\boldsymbol{x}^{(N)}\in X^{N}$ and $\epsilon>0$
\begin{equation}
e^{-\beta NE^{(N)}(\boldsymbol{x}^{(N)})}\leq A_{\epsilon}e^{CN\epsilon}\frac{\int_{\delta_{N}^{-1}\left(B_{\epsilon}(\delta_{N}(\boldsymbol{x}^{(N)})\right)}e^{-\beta NE^{(N)}}dV^{\otimes N}}{\int_{\delta_{N}^{-1}\left(B_{\epsilon^{2}}(\delta_{N}(\boldsymbol{x}^{(N)})\right)}dV^{\otimes N}}\label{eq:as submean ineq}
\end{equation}
 when the metric $d$ on $\mathcal{P}(X),$ defining the weak topology,
is taken to be the Wasserstein $L^{2}-$metric $d_{W^{2}}$. Combining
the previous inequality with the Gamma-convergence \ref{eq:Gamma conv towards E text}
it is straightforward to conclude the proof of the LDP in Theorem
\ref{thm:LDP beta pos text} (see \cite{berm8} and the exposition
in \cite{berm10b}).

\subsubsection{The proofs of the technical ingredients}

We briefly recall the proofs of the two technical ingredients discussed
above. First, an important ingredient in the proof of the asymptotics
\ref{eq:asympt of log det L infty} in \cite{b-b} is the simple observation
that the $L^{\infty}(X^{N})$ norm in question has the same logarithmic
asymptotics as the corresponding $L^{2}(X^{N})-$norm. In turn, the
$L^{2}(X^{N})-$norm may, by expanding $\det S^{(k)}(x_{1},...x_{N})$
as an alternating sum over the $N!$ elements of $S_{N},$ be expressed
as 
\begin{equation}
\left\Vert \det S^{(k)}e^{-ku}\right\Vert _{L^{2}(X^{N_{k}})}^{2}=N_{k}\det_{i,j\leq N_{k}}A_{ij}[u],\label{eq:L two norm of det as G-S}
\end{equation}
 where $A_{ij}[u]$ is the $N\times N$ Gram matrix defined by the
scalar products of the base elements $s_{i}^{(k)}$ in $H^{0}(X,kL)$
with respect to the scalar product induced by the volume form $dV$
and the metric $\left\Vert \cdot\right\Vert e^{-u/2}$ on $L.$ As
shown in \cite{b-b} the corresponding convergence towards $\mathcal{E}_{\omega_{0}}(Pu)$
then follows from Bergman kernel asymptotics on $X,$ using that the
differential of the functional
\begin{equation}
u\mapsto-\frac{1}{kN_{k}}\log\left\Vert \det S^{(k)}e^{-ku}\right\Vert _{L^{2}(X^{N_{k}})}^{2}\label{eq:function log of L two norm}
\end{equation}
 on $C(X)$ is represented by the probability measure on $X$ defined
by the point-wise norm of the Bergman kernel on the diagonal, with
respect to the metric $\left\Vert \cdot\right\Vert e^{-u/2}$ on $L.$
\begin{rem}
Formula \ref{eq:L two norm of det as G-S} allows one to view the
functional in formula \ref{eq:function log of L two norm} as an analogue
of Donaldson's $\mathcal{L}-$functional \cite{do1b}. However, in
contrast to the setting in \cite{do1b}, it is crucial that $u$ is
allowed to be any continuous (or smooth) function and not only a function
in $\mathcal{H}(X,\omega_{0}).$ 
\end{rem}

Finally, we recall the the starting point of the proof of the asymptotic
submean inequality \ref{eq:as submean ineq} in \cite{berm8} is the
well-known fact that the embedding $\delta_{N}$ of $X^{N}/S_{N}$
into the $L^{2}-$Wasserstein space $(\mathcal{P}(X),d_{W_{2}})$
is an isometry when $X^{N}/S_{N}$ is endowed with the quotient space
(orbifold) metric induced from the Riemannian metric $g_{N}$ on $X^{N},$
defined as $N^{-1}$ times the product Riemannian metric. The quasi-subharmonic
property of $NE^{(N)}$ is equivalent to 
\[
\Delta_{g_{N}}E^{(N)}\geq-\lambda
\]
 on $X^{N}.$ Moreover, the scaling of $g_{N}$ also ensures that
the Ricci curvature of $g^{(N)}$ is bounded from below by a uniform
constant times the dimension of $X^{N}.$\emph{ }The inequality \ref{eq:as submean ineq}
now follows from the general sub-mean inequality in \cite[Thm 2.1]{berm8}
for Riemannian quotients (orbifolds) $Y:=M/G$ (which yields a distortion
factor with sub-exponential growth in the dimension). We recall that
the latter inequality is proved using geometric analysis on the orbifold
$Y,$ by generalizing an inequality of Li-Schoen in Riemannian geometry
\cite{li-sc}. A key ingredient in the proof is the Cheng-Yau gradient
estimate \cite{c-y} for harmonic functions on a Riemannian manifold
(or more generally orbifold) and the observation that the dependence
on the dimension in the estimate is sub-linear. 

\subsection{\label{subsec:A-general-LDP}A general LDP and the Gärtner-Ellis
theorem}

The same method of proof, in fact, yields the following general LDP:
\begin{thm}
\label{thm:LDP subharmon}Let $E^{(N)}$ be a sequence of lower semi-continuous
symmetric functions on $X^{N},$ where $X$ is a compact Riemannian
manifold.\textup{ Assume that }

\begin{itemize}
\item \textup{The corresponding functions $E^{(N)}$ on $\mathcal{P}(X)$
converge to a functional $E,$ in the sense of Gamma$-$convergence
on $\mathcal{P}(X).$}
\item $NE^{(N)}$ is uniformly quasi-superharmonic, i.e. \emph{$\Delta_{x_{1}}NE^{(N)}(x_{1},x_{2},...x_{N})\leq C$
on $X^{N}$}
\end{itemize}
Then, for any sequence of positive numbers $\beta_{N}\rightarrow\beta\in]0,\infty]$
the measures $\Gamma_{N}:=(\delta_{N})_{*}e^{-\beta_{N}NE^{(N)}}$
on $\mathcal{P}(X)$ satisfy, as $N\rightarrow\infty,$ a LDP with
\emph{speed} $\beta_{N}N$ and \emph{rate functional} 
\begin{equation}
F_{\beta}(\mu)=E(\mu)+\frac{1}{\beta}D_{dV}(\mu)\label{eq:free energy func theorem gibbs intro-1}
\end{equation}
Moreover, assuming that the second point above holds, the first point
may be replaced by the following assumption: there exists a sequence
$\beta_{N}\rightarrow\infty$ such that for any $u\in C^{0}(X)$ \textup{
\begin{equation}
\mathcal{F}_{\beta_{N}}(u):=-\frac{1}{N\beta_{N}}\log Z_{N,\beta}[u],\,\,Z_{N,\beta}[u]:=\int_{X^{N}}e^{-\beta_{N}\left(NE^{(N)}+u\right)}dV^{\otimes N},\label{eq:conv in g-e}
\end{equation}
 converges, as $N\rightarrow\infty,$ towards $\mathcal{F}(u)$ for
some Gateaux differentiable functional $\mathcal{F}$ on $C^{0}(X)$.
}Then the Gamma-convergence in the first point above holds with $E$
defined as the Legendre-Fenchel transform of $f(u):=-\mathcal{F}(-u).$
\end{thm}

To see the connection between the last statement in the previous theorem
and the present complex-geometric setup note that the functional in
formula \ref{eq:function log of L two norm} can be expressed as $\mathcal{F}_{\beta_{N}}(u)$
for $\beta_{N_{k}}=k.$ In this particular case, the corresponding
random point processes is a \emph{determinantal point process} (see
\cite{berm10b,du} for background on such processes and the relations
to Fekete points and interpolation nodes).

As explained in \cite{berm8}, the previous theorem can be viewed
as generalization of the Gärtner-Ellis theorem to $\beta\in]0,\infty[.$
The Gärtner-Ellis theorem (which is a generalization of Cramér's classical
LDP theorem for independent random vectors) says, when applied to
the laws of Gibbs measures, that the last assumption in Theorem \ref{thm:LDP subharmon}
implies that LDP holds for $\beta=\infty.$ However, extending the
LDP to $\beta\in]0,\infty[$ appears to require assumptions on the
nature of $E^{(N)},$ such as the superharmonicity assumption in Theorem
\ref{thm:LDP subharmon}. 

\section{\label{sec:Towards-the-case of neg}Towards the case of negative
$\beta$ }

Now consider the ``Fano setting'' with $\beta\in[-1,0[$ (see Section
\ref{subsec:The-Fano-setting}). In order to extend the method of
proof discussed in section to the case when $\beta<0$ it seems natural
to expect that one would need to exploit that $\beta E^{(N}$ is uniformly\emph{
}quasi-plurisubharmonic. One small step in this direction is taken
in the following
\begin{lem}
There exists $\beta_{0}<0$ such that for any $\beta>\beta_{0}$ the
following bound holds for a positive constant $C_{\beta}:$ 
\[
N^{-1}\log Z_{N,\beta}\leq C_{\beta}
\]
\end{lem}

\begin{proof}
Setting $\varphi^{(N)}:=-E^{(N)}$ on $X^{N},$ the functions on $X$
obtained by fixing all but one arguments in $\varphi^{(N)}$ are $\omega_{0}-$psh
on $X.$ By \ref{eq:asympt of log det L infty} there exists a uniform
constant $C$ such that $\sup_{X^{N}}\varphi^{(N)}\leq C_{0}.$ Moreover,
as is well-known, there exists a positive number $\alpha$ such that
for any $\gamma<\alpha$ there exists a constant $A_{\gamma}$ such
that $\int_{X}e^{-\gamma\varphi}\leq A_{\gamma}e^{-\gamma\sup_{X}\varphi}$
for any $\varphi\in PSH(X,\omega_{0}).$ Indeed, the optimal such
$\alpha$ is Tian's \emph{$\alpha-$invariant} of $c_{1}(-K_{X})$
aka as the \emph{global log canonical threshold} of $X$ (see, for
example, the appendix in \cite{berm6}, which applies to more general
reference measure $\mu_{0}$). Hence, the lemma follows with $\beta_{0}=-\alpha$
and $C_{\beta}=\log(A_{\gamma})+\gamma C_{0},$ by writing $Z_{N,-\gamma}$
as an iterated integral over the $N$ factors of $X^{N}.$
\end{proof}
Here we will, however, propose a different \emph{variational }route,
based on Gibbs variational principle. As explained in \cite{berm11}
this approach is successful in the one-dimensional setting, but, in
general, it hinges on a missing energy bound.

It will be convenient to consider the setting of a general $\beta\in]0,1].$
In order to get started we will make the following assumptions: 
\begin{itemize}
\item The normalizing constant $Z_{N,\beta}$ is finite, i.e. the corresponding
Gibbs measure $\mu_{\beta}^{(N)}$ is well-defined (formula \ref{eq:def of Gibbs measure Fano setting intro}).
\item The free energy functional $F_{\beta}$ on $\mathcal{P}(X)$ has a
unique minimizer $\mu_{\beta}$ 
\end{itemize}
The goal is then to prove that the random measure $\delta_{N}$ on
$(X^{N},\mu_{\beta}^{(N)})$ converges in law towards the unique minimizer
$\mu_{\beta}$ of $F_{\beta},$ i.e. that the convergence
\begin{equation}
\Gamma_{N,\beta}:=(\delta_{N})_{*}\mu_{\beta}^{(N)}\rightarrow\delta_{\mu_{\beta}},\,\,\,N\rightarrow\infty\label{eq:goal conv of laws}
\end{equation}
holds in the weak topology on
\[
\mathcal{X}:=\mathcal{P}(\mathcal{Y}),\,\,\,\mathcal{Y}:=\mathcal{P}(X)
\]
 We start by recalling \emph{Gibbs variational principle}, which is
a standard tool in Statistical Mechanics, involving the\emph{ $N$-particle
mean free energy functional $F_{\beta}^{(N)}$} on the space $\mathcal{P}(X^{N})$
of all probability measures on $X^{N},$ defined by

\begin{equation}
F_{\beta}^{(N)}(\mu_{N}):=\beta E^{(N)}(\mu_{N})+D^{(N)}(\mu_{N}),\label{eq:def of mean free energy}
\end{equation}
 where $E^{(N)}(\mu_{N})$ denotes the \emph{$N-$particle mean free
energy} 
\[
E^{(N)}(\mu_{N}):=\int_{X^{N}}E^{(N)}\mu_{N},
\]
 and $D^{(N)}(\mu_{N})$ denotes \emph{the $N-$particle mean entropy}
(relative to $dV^{\otimes N})$ 
\[
D^{(N)}(\mu_{N}):=D_{dV^{\otimes N}}(\mu^{(N)})/N
\]

\begin{lem}
\label{lem:(Gibbs-variational-principle).}(Gibbs variational principle).
Assume that $Z_{N,\beta}<\infty.$ Then the Gibbs measure $\mu_{\beta}^{(N)}$
is the unique minimizer of the functional $F_{\beta}^{(N)}$ on $\mathcal{P}(X^{N}).$
Moreover, 
\[
-\log Z_{N,\beta}=\inf_{\mathcal{P}(X^{N})}F_{\beta}^{(N)}
\]
As a consequence, when $E^{(N)}$ is symmetric, i.e. $S_{N}-$invariant
$\mu_{\beta}^{(N)}$ is the unique minimizer of $F_{\beta}^{(N)}$
on the space $\mathcal{P}(X^{N})^{S_{N}}$ of all $S_{N}-$invariant
probability measures on $X^{N}$ and
\[
-\log Z_{N,\beta}=\inf_{\mathcal{P}(X^{N})^{S^{N}}}F_{\beta}^{(N)}
\]
\end{lem}

\begin{proof}
This follows directly from rewriting $F_{\beta}^{(N)}(\mu_{N})=-\log Z_{N,\beta}+D_{\mu_{\beta}^{(N)}}(\mu_{N})$
and using that $D_{\nu}(\mu)\geq0$ with equality iff $\mu=\nu$
\end{proof}
Next, in order to study the limit $N\rightarrow\infty$ we embed all
the spaces $\mathcal{P}(X^{N})^{S^{N}}$ into the space $\mathcal{X}:$
\[
(\delta_{N})_{*}:\,\mathcal{P}(X^{N})^{S^{N}}\rightarrow\mathcal{X}
\]
We can then identify mean free energies $F^{(N)}$ with functionals
on $\mathcal{X},$ extended by $\infty$ to all of $\mathcal{X}.$
In particular, this means that we identity $\mathcal{P}(X)$ with
its image in $\mathcal{X}$ under the embedding $\mu\mapsto\delta_{\mu}.$
Consider now the following functional on $\mathcal{X}:$ 

\[
F_{\beta}(\Gamma)=\beta E(\Gamma)+D(\Gamma),
\]
defined when $E(\Gamma)<\infty,$ where $E(\Gamma)$ and $D(\Gamma)$
are the affine functionals on $\mathcal{X}$ defined by
\[
E(\Gamma):=\int E(\mu)\Gamma,\,\,\,D(\Gamma):=\int D(\mu)\Gamma
\]
In the case when $E(\Gamma)=\infty$ we define $F_{\beta}(\Gamma):=\infty.$
In order to prove the weak convergence \ref{eq:goal conv of laws}
it is, in view of the previous lemma, enough to show the following
conjectural convergence
\begin{equation}
\lim_{N\rightarrow\infty}F_{\beta}^{(N)}=F,\,\,\,\text{on\ensuremath{\,\mathcal{X}}}\label{eq:conj gamma conv}
\end{equation}
in the sense of Gamma-convergence relative to some subset $\mathcal{S}$
of $\mathcal{X}$ containing the minima of $F.$ Indeed, by Gibbs
variational principle, if the previous convergence holds, then by
Lemma \ref{lem:conv of inf} it is enough to show that the affine
functional $F$ has the following property: it has a unique minimum
and moreover the minimum is attained at $\delta_{\mu_{\beta}}.$ But
this follows from the following two results. First, we have the following
elementary lemma:
\begin{lem}
\label{lem:concave F}Suppose that $F(\mu)$ is lsc on $\mathcal{P}(X)$
and admits a unique minimizer $\mu_{*}.$ Then $\delta_{\mu_{*}}$
is the unique minimizer of the affine functional $F(\Gamma)$ on $\mathcal{P}(X).$ 
\end{lem}

\begin{proof}
Since $F(\mu)$ is lsc we can write $F_{\beta}(\Gamma)=\int F_{\beta}(\mu)\Gamma,$
which is lsc and affine on $\mathcal{X}.$ Since $F_{\beta}(\Gamma)$
is affine we have
\[
\inf_{\mathcal{X}}F=\inf_{P(X)}F=F(\mu_{*})
\]
After shifting $F(\mu)$ by a constant we may as well assume that
$F(\mu_{*})=0.$ Take $\Gamma\neq\delta_{\mu_{*}}.$ Then there exists
a compact subset $K$ of $\mathcal{P}(X),$ not containing $\mu_{*}$
and such that $\Gamma(K)>0.$ Moreover, since $F$ is lsc on $\mathcal{P}(X)$
we have $F(\mu)\geq\delta$ on $K$ for some $\delta>0.$ But then
$F(\Gamma)\geq\delta\Gamma(K)>0=F(\mu_{*}),$ which concludes the
proof. 
\end{proof}
Secondly we have the following 
\begin{lem}
The functional $F_{\beta}$ is lsc on $\mathcal{P}(X)$ and hence
so is its affine extension to $\mathcal{X}$
\end{lem}

\begin{proof}
Theorems \ref{thm:free energy Fano} and \ref{thm:free energy beta minus one Fano}
show that, in fact, $F_{\beta}$ is lsc iff it admits a unique minimizer,
which we have assumed.
\end{proof}
\begin{rem}
\label{rem:LDP vs Gamma}By general principles (see \cite{berm10})
the strong LDP form of Conjecture \ref{conj:Fano with triv autom intr},
formulated in Conjecture \ref{conj:LDP fano intro} is, in fact, equivalent
to having bona fide Gamma-convergence in \ref{eq:conj gamma conv}.
But using the weaker \emph{relative} notion may have some advantages,
as discussed in Step 2 below. 
\end{rem}

We next take a first step towards proving the relative Gamma-convergence
\ref{eq:conj gamma conv}:

\subsubsection*{Step 1: The existence of a recovery sequence}

By Lemma \ref{lem:concave F} it is enough to prove the existence
of a recovery sequence $\Gamma_{N}$ for any $\Gamma$ of the form
$\Gamma=\delta_{\mu}.$ To this end we set $\Gamma_{N}:=(\delta_{N})_{*}\mu^{\otimes N}$
and first observe that the mean entropy is additive in the following
sense: 
\[
D^{(N)}(\mu^{\otimes N})=D(\mu),
\]
 for any given $\mu\in\mathcal{P}(X).$ Indeed, this follows directly
from the additivity of log. Next observe that
\[
\liminf_{N\rightarrow\infty}E^{(N)}(\mu^{\otimes N})\geq E(\mu)
\]
Indeed, fixing $u\in C^{0}(X)$ and rewriting 
\[
E^{(N)}(\mu^{\otimes N})=\int_{X^{N}}\left(E^{(N)}(x_{1},...,x_{N})+N^{-1}\sum_{i=1}^{N}u(x_{i})\right)-\int_{X}u\mu
\]
 and estimating the integral over $X^{N}$ from below by its infimum,
this follows directly from combining the asymptotics \ref{eq:asympt of log det L infty}
with \ref{prop:pluri energy as legendre}. Hence, since $\beta<0,$
this shows that 
\[
\limsup_{N\rightarrow\infty}F_{\beta}^{(N)}(\Gamma_{N})\leq F_{\beta}(\Gamma),
\]
 as desired.

\subsubsection*{Towards the missing Step 2: the lower bound}

We first recall the fundamental fact that mean entropy $D^{(N)}$
on $\mathcal{X}$ satisfies the lower bound in the definition of Gamma-convergence
(as follows from basic sub-additive properties of the entropy; see
\cite{r-r} and \cite[Thm 5.4]{h-m}). In particular, 
\[
D(\Gamma)\leq\liminf_{N\rightarrow\infty}D^{(N)}(\mu_{\beta}^{(N)})<\infty
\]
Since $\beta<0,$ it would thus be enough to show the following ``upper
bound property of the mean energy'':
\begin{equation}
\limsup_{N\rightarrow\infty}E^{(N)}(\mu_{\beta}^{(N)})\leq E(\Gamma_{\beta}):=\int E(\mu)\Gamma_{\beta}(\mu)\label{eq:upper bound property}
\end{equation}
for any limit point $\Gamma_{\beta}$ in $\mathcal{X}$ of $\Gamma_{N,\beta}.$
It should, however, be stressed that such bound can not hold for \emph{any}
sequence in $\mu^{(N)}$ in $\mathcal{P}(X^{N})^{S_{N}},$ since $E(\mu)$
is  not continuous on $\mathcal{P}(X)$. Note that, in general, $D(\Gamma)<\infty\implies E(\Gamma)<\infty,$
since $F_{\beta}(\Gamma)$ is bounded from below on $\mathcal{X}$
for some $\beta<-1.$ 

We summarize the output of the previous discussion, applied to the
case $\beta=-1,$ in the following
\begin{thm}
\label{thm:Fano setting text}Let $X$ be a Fano manifold and assume
that $X$ is uniformly Gibbs stable. Then $X$ admits a unique Kähler-Einstein
metric $\omega_{KE}$ and if the ``upper bound property of the mean
energy'' \ref{eq:upper bound property} holds for the canonical sequence
$\mu^{(N)},$ then the empirical measures $\delta_{N}$ of the canonical
random point process on $X$ converge in law towards the normalized
volume form $dV_{KE}$ of $\omega_{KE}.$
\end{thm}

\begin{proof}
Set $\beta=-1.$ If $X$ is uniformly Gibbs stable, then as shown
in \cite{f-o}, $X$ is uniformly K-stable. Hence, by either \cite{c-d-s}
or \cite{bbj} $X$ admits a unique Kähler-Einstein metric. Thus $F_{-1}$
admits a unique minimizer, $\mu_{-1},$ given by $dV_{KE}.$ Moreover,
by the uniform Gibbs stability $Z_{N,\beta}$ is finite for $N$ sufficiently
large. Hence, the theorem follows from Step 1 and 2 above.
\end{proof}
\begin{rem}
Note that if the ``upper bound property of the mean energy'' holds
then the argument above shows that
\[
\lim_{N\rightarrow\infty}\inf_{\mathcal{P}(X^{N})^{S_{N}}}F_{\beta}^{(N)}=\lim_{N\rightarrow\infty}\inf_{\mathcal{P}(X)^{\otimes N}}F_{\beta}^{(N)},
\]
 i.e. asymptotically, as $N\rightarrow\infty$ the infimum of the
mean free energy functional $F_{\beta}^{(N)}$ can be restricted to
the subspace $\mathcal{P}(X)^{\otimes N}\Subset\mathcal{P}(X^{N})^{S_{N}}.$ 
\end{rem}

\subsection{\label{subsec:Analyticity-and-absence}Analyticity and absence of
phase transitions}

The proof of Theorem \ref{thm:Fano setting text} reveals that in
order to establish the convergence towards $dV_{KE}$ it is enough
to show that 
\begin{equation}
-\lim_{N\rightarrow\infty}\frac{1}{N}\log Z_{N.\beta}=\inf_{\mathcal{P}(X)}F_{\beta}\label{eq:conv log Z N minus one intro-1}
\end{equation}
 for $\beta=-1.$ By Theorem \ref{thm:free energy beta pos} the converge
does hold for $\beta\geq0$ and the problem of extending the convergence
to $\beta=-1$ can be connected to the theory of phase transitions
in statistical mechanics. To see this first consider the following
general setup. Let $H^{(N)}$ be a sequence of measurable functions
(``Hamiltonians'') on the measure spaces $(X_{N},dV_{N})$ such
that the corresponding partition function
\[
Z_{N,\beta}:=\int_{X}e^{-\beta H^{(N)}}dV_{N}
\]
is finite for some $\beta>\beta_{0}.$ Then $Z_{N,\beta}$ is real-analytic
in $\beta$ on $]-\beta_{0},\infty[$ for any $N$ and strictly positive.
However, in general $N^{-1}\log Z_{N,\beta}$ may converge to a function
which is\emph{ not} real-analytic. This is often taken as the definition
of a \emph{phase transition }in statistical mechanics\emph{ }(see
\cite[Chapter 5]{ru}). 
\begin{example}
The prime example of a phase transition is provided by the Curie-Weiss
mean field model for magnetization in spin systems, where $N$ is
the number of spins. If the sign convention for the corresponding
Hamiltonians is taken so that $H^{(N)}$ is anti-ferromagnetic then
the real-analyticity in question brakes down at a \emph{critical }negative
inverse temperature $\beta_{c}$ (in this case $\beta_{0}=-\infty).$
This is precisely the inverse temperature for which the convergence
of the empirical magnetization towards a deterministic limit fails
(see the appendix in \cite{berm10b} for a comparison between the
Curie-Weiss model and the present complex-geometric setup).
\end{example}

In the present Fano setting the partition function $Z_{N,\beta}$
is real-analytic in $\beta,$ as long as $-\beta$ is strictly smaller
than the stability threshold $\gamma_{N}$(formula \ref{eq:def of gamma N k}).
In fact, according to well-known results of Atiyah and Bernstein-Gelfand
the Archimedean zeta function $\beta\mapsto Z_{N,\beta}$ extends
to a meromorphic function on $\C$ with a discrete set of rational
poles located at $]-\infty,0[\subset\C$ (see the book \cite{ig}).
The next result shows that the convergence in Conjecture \ref{conj:Fano with triv autom intr}
holds in the absence of phase transitions down to the inverse temperature
$-1$:
\begin{thm}
\label{thm:phase}Assume that there exists $\beta_{0}<-1$ and a real-analytic
function $f(\beta)$ on $]-\beta_{0},\infty[$ such that for any\emph{
$\beta\in]-\beta_{0},0[$ }
\[
-\lim_{N\rightarrow\infty}\frac{\log Z_{N,\beta}}{N}=f(\beta).
\]
 Then $X$ admits a unique Kähler-Einstein-metric $\omega_{KE}$ and
the empirical measures $\delta_{N}$ of the canonical random point
process on $X$ converge in law towards the normalized volume form
$dV_{KE}$ of $\omega_{KE}.$
\end{thm}

\begin{proof}
\emph{Step 1: For $\beta\in[-1,\infty[$ the free energy functional
$F_{\beta}$ admits a unique minimizer $\mu_{\beta}$ on $\mathcal{P}(X)$ }

To see this first recall that the argument in Step 1 of the proof
of Theorem \ref{thm:Fano setting text} shows that \emph{for $\beta\in[-1,\infty[$
we have} $F_{\beta-\epsilon}\geq-C$ for some positive constants $C$
and $\epsilon$ (depending on $\beta).$ Hence, by Theorems \ref{thm:F lsc for beta bigger}
$F_{\beta}$ admits a unique minimizer $\mu_{\beta}.$ In particular,
there exists a unique Kähler-Einstein metric on $X.$ 

\emph{Step 2: The function $F(\beta):=F(\mu_{\beta})$ is real-analytic
on $]-1,\infty[$ and continuous up to the boundary at $\beta=-$1}

The continuity up to $\beta=-1$ was shown in\cite{b-m}. Now fix
positive integers $p,l$ and consider the Banach (Hilbert) spaces
defined by the Sobolev spaces $\mathcal{B}_{1}:=L^{p,l+2}(X)$ and
$\mathcal{B}_{2}:=L^{p,l}(X).$ Take $p$ and $l$ sufficiently large
so that $\mathcal{B}_{2}\subset C^{2}(X)$ (as ensured by the Sobolev
inequality). Consider the map 
\[
g:\,\mathcal{B}_{1}\times]-\beta_{0},\infty[\rightarrow\mathcal{B}_{2}\times\{0\},\,\,\,(u,\beta)\mapsto\left((dd^{c}u)^{n}/dV-e^{\beta u},\int_{X}\beta^{-1}(e^{\beta u}-1)dV\right)
\]
(where $\beta^{-1}(e^{\beta t}-1)$ is defined to be equal to $t$
for $\beta=0).$ The definition of $g$ is made so that $\varphi_{\beta}$
solves the MA-equation \ref{eq:asympt of log det L infty} iff $g(\varphi_{\beta},\beta)=(0,0)$
and when $\beta=0$ the solution $\varphi_{\beta}$ is normalized
so that $\int\varphi_{\beta}dV=0.$ The map $g$ is a real-analytic
map between Banach spaces in the sense of \cite{to}. Indeed, $u\mapsto(dd^{c}u)^{n}/dV$
is continuous and multilinear (of order $2n)$ and the functions $e^{\beta t}$
and $\beta^{-1}(e^{\beta t}-1)$ are both real-analytic on $\R\times\R.$
Next, note that the directional derivative $D_{u}(g,u)$ is surjective
for $\beta\in[-1,\infty[$ and $u\in\mathcal{H}(X).$ Indeed, this
is shown in the course of the proof of the openness property in Aubin's
continuity path \cite{au} (see also \cite{b-m}). Hence, it follows
from the real-analytic implicit function theorem in Banach spaces
\cite{to}, that the curve $\varphi_{\beta}$ is real-analytic in
$\mathcal{B}_{1}.$ But then it follows from the explicit expression
\ref{eq:beauti E as integral} for $\mathcal{E}(\varphi),$ which
is a sum of multilinear terms in $\varphi$ that the function
\[
E(\beta):=E_{\omega_{0}}(\mu_{\beta})=\mathcal{E}_{\omega_{0}}(\varphi_{\beta})-\left\langle \varphi_{\beta},MA(\varphi_{\beta}\right\rangle 
\]
 is real-analytic on $]-1,\infty[.$ The proof of Step 2 is now concluded
by observing that $dF(\beta)/d\beta=E(\beta).$ Indeed, this follows
from the chain rule using that $\mu_{\beta}$ is the unique minimizer
of $F_{\beta}$ on $\mathcal{P}(X).$ 

\emph{Step 3: Conclusion of proof}

By Theorem \ref{thm:LDP beta pos text} $f(\beta)=F(\beta)$ for $\beta>0$
(note that the relation $f(0)=F(0)$ is trivial). Since $f$ and $F$
are both real-analytic on\emph{ $]-1,\infty[$ }and continuous as
$\beta\rightarrow-1$ it follows that $f(\beta)=F(\beta)$ on all
of \emph{$[-1,\infty[.$ }Hence, it follows from the proof of Theorem
\ref{thm:Fano setting text} that for any given $\beta\in[-1,\infty[$
the empirical measures $\delta_{N}$ on $(X^{N},\mu_{\beta}^{(N)})$
converge in law towards $\mu_{\beta}.$ Indeed, since $f(\beta)=F(\mu_{\beta})$
the assumed convergence can be used as a replacement for Step 2 in
the proof of Theorem \ref{thm:Fano setting text}. Specializing to
$\beta=-1$ thus concludes the proof. 
\end{proof}
Assuming the existence of a limiting function $f(\beta),$ one way
of establishing the real-analyticity of $f$ is to show that the meromorphic
extension of $Z_{N,\beta}$ to $\C$ is holomorphic and has no zeroes
on some $N-$independent neighborhood of $[-1,0]$ in $\C$ (using
Montel's convergence theorem of complex analysis). This is an approach
that was pioneered by Lee-Yang for some statistical mechanical models,
including spin models and lattice gases \cite{y-l}. Interestingly,
in the present complex-geometric setup such a non-vanishing result
holds in the setting of log Fano curves (in a slightly different setup),
as shown in \cite{berm11}. However, the general case is completely
open.

\end{document}